\documentclass{amsart}

\usepackage{amsfonts}
\usepackage{amscd,amsmath,amsthm,amssymb}
\usepackage{bm}
\usepackage{verbatim}
\usepackage{lscape}
\usepackage{geometry}
\newcommand{\R}{\mathbb R}

\newcommand{\Z}{\mathbb Z}
\newcommand{\HH}{{\mathbb{H}}}

\def\a'{\`a}
\def\e'{\`e}
\def\o'{\`o}
\def\u'{\`u}

\begin{document}

\title{$Sp(n)$-orbits of isoclinic subspaces in the real Grassmannians.}
\author{Massimo Vaccaro}
\address{Dipartimento dell'Ingegneria di Informazione e Matematica Applicata, Universit\a' di Salerno, 84084 - Fisciano (SA) , Italy}
\email{massimo\_vaccaro@libero.it}
\thanks{Work done under the programs of GNSAGA-INDAM of C.N.R. and PRIN07 "Riemannian metrics and differentiable structures" of MIUR (italy)}
\date{\today} 

\maketitle

\markboth{ \rm{MASSIMO VACCARO}} {\rm{$Sp(n)$-orbits of isoclinic subspaces in the real Grassmannians}}


\newtheorem{teor}{Theorem}[section]
\newtheorem{coro}[teor]{Corollary}
\newtheorem{lemma}[teor]{Lemma}
\newtheorem{prop}[teor]{Proposition}
\newtheorem{samp}[teor]{Example}
\newtheorem{defi}[teor]{Definition}
\newtheorem{exer}[teor]{Exercise}
\newtheorem{remk}[teor]{Remark}
\newtheorem{conj}[teor]{Conjecture}
\newtheorem{claim}[teor]{Claim}

\newtheorem{intteo}{Theorem}
\newtheorem{intcoro}{Corollary}

\section{Abstract}

In the framework of the study of the $Sp(n)$-orbits in the real Grassmannian $G^\R(k,4n)$ of $k$-dimensional non oriented subspaces of a real $4n$-dimensional vector space $V$, here we consider the case of  the isoclinic subspaces  whose set we indicate with $\mathcal{IC}$. Endowed $V$ with an Hermitian quaternionic structure $(\mathcal{Q},<,>)$, a subspace $U$ is isoclinic if for any compatible complex structure  $A \in \mathcal{Q}$ the principal angles of the pair $(U,AU)$ are all the same, say $\theta^A$. We will show that, fixed an admissible hypercomplex basis $(I,J,K)$, to any such subspace $U$ we can associate two set of invariants, namely a triple $(\xi,\chi,\eta)$ and a pair $(\Gamma, \Delta)$ where  $\Gamma$ itself is a function of $(\xi,\chi,\eta)$. We prove that the angles of isoclinicity $(\theta^I,\theta^J,\theta^K)$ together with $(\xi,\chi,\eta, \Delta)$ determine its $Sp(n)$-orbit.  In particular if $\dim U= 8k+2$ or $\dim U= 8k+6$ with $k \geq 0$ the last set reduce to  the pair $(\xi= \pm 1,\chi= \pm 1)$.

\section{Summary}
In   \cite{Vacpreprint}
we determined the set of invariants characterizing the $Sp(n)$-orbits in the Grassmannian $G^\R(k,4n)$ of $k$-dimensional non oriented  real subspaces of a $4n$-dimensional real vector space $V$. There we endowed $V$ with a quaternionic structure $\mathcal{Q}$, an Hermitian metric $<,>$  and denoted by $S(\mathcal{Q})$ the sphere of compatible complex structures of $\mathcal{Q}$.
 Given  a $k$-dimensional subspace $U \subseteq V$,  
 for any $A \in  S(\mathcal{Q})$ we denoted by $\omega^A=<X,AY>, \; X,Y \in U $ the skew-symmetric $A$-K\"{a}hler form restricted to $U$.  Such  form assumes a \textit{standard form} represented   w.r.t. some  orthonormal basis $\{X_i \}$ by the skew-symmetric matrix
 \[
(\omega^A_{ij})=(<X_i,AX_j>)=
  \left\{
 \begin{array} {l}
 \geq 0  \quad  \mbox{  if $i$ is odd  and  } j=i+1,\\
     0  \quad \mbox{otherwise},
 \end{array}
 \right.
 \]
 for $i \leq j \leq k$.
 We call $\{X_i \}$ a \textit{standard basis} of $\omega^A|_U$ and, for $i$ odd,  $i < k$, the linear span $L(X_i,X_{i+1})$ a \textit{standard 2-plane}. By the skew symmetry of $\omega^A$, the standard bases are never unique (even if $\dim U=1$ they are defined up to sign). Moreover, by the assumption that the quantities $\omega^A_{ij}$ are non negative,  for $i$ odd  and  $j=i+1 \leq k$ one has that the entries  $\omega^A_{ij}$ are   the cosines of the  principal angles $\theta_{ij} \in [0,\pi/2]$  of the pair of 2-planes  $(L(X_i,X_{i+1}),AL(X_i,X_{i+1}))$. Moreover we call \textit{ $\omega^A$-standard subspaces} the uniquely determined subspaces $U_i \subset U$  associated to the same principal angle.

 Chosen an admissible (hypercomplex) basis $(I,J,K)$ of $\mathcal{Q}$, we denoted by $\mathcal{B}(U)$ the set of triples of standard bases of the skew-symmetric forms
 $\omega^I|_U,\omega^J|_U,\omega^K|_U$.   Necessary and sufficient conditions  for a pair of subspaces $U,W$ to belong to the same $Sp(n)$-orbit are  stated in the
  \begin{teor}\cite{Vacpreprint} \label{main_theorem 1_ existance of canonical bases with same mutual position}
Let $(I,J,K)$ be an admissible  basis of $\mathcal{Q}$. The non oriented subspaces  $U^m$ and $W^m$ of $V$ are in the same $Sp(n)$-orbit iff
    \begin{enumerate}
    \item   they share the same  $I,J,K$ principal angles i.e.
      \[\theta^I(U)=\theta^I(W),\quad \theta^J(U)=\theta^J(W),\quad \theta^K(U)=\theta^K(W)\] for one and hence any admissible  basis $(I,J,K)$ or equivalently      the singular values of the projectors  $Pr^{IU}: U \rightarrow IU, \; Pr^{JU},Pr^{KU}$ equal those of $Pr^{IW}:W \rightarrow IW, \; Pr^{JW},Pr^{KW}$ for one and hence any admissible basis $(I,J,K)$.
    \item there exist three (orthonormal) standard bases $(\{ X_i \} , \{ Y_i \} ,\{ Z_i \}) \in \mathcal{B}(U)$   and  $(\{ X'_i \},\{ Y'_i \},\{ Z'_i \}) \in \mathcal{B}(W)$ whose relative position is the same or equivalently  \[C_{IJ}= C_{IJ}', \qquad C_{IK}= C_{IK}'\]
          where $C_{IJ}=(<X_i,  Y_j>), \qquad C_{IJ}'=(< X_i', Y_j'>), \qquad  C_{IK}=(<X_i, Z_j>),\qquad   C_{IK}'=(<X_i', Z_j'>)$.
     \end{enumerate}
\end{teor}

The determination of the principal angles between a pair of subspaces $S,T$ is a well known problem solved by the singular value decomposition of the orthogonal projector of $S$ onto $T$.
The problem to determine the $Sp(n)$-orbits  in $G^\R(k,4n)$ of $k$-dimensional real subspaces  for $k=1,\ldots,4n$ , turns then into the one of determining the existence of a triple of bases belonging to $\mathcal{B}(S)$ and $\mathcal{B}(T)$,  satisfying the condition  2) of the Theorem (\ref{main_theorem 1_ existance of canonical bases with same mutual position}). This is not an easy problem to solve because, as already said,  for any $A \in S(\mathcal{Q})$, the standard bases of $\omega^A$  are never unique.   An iterative  procedure to determine a \textit{canonical form} of the matrices $C_{IJ},C_{IK}$ of a subspace $U$ (that we denoted \textit{canonical matrices}) is given in   \cite{Vacpreprint}.

 The idea is to refer $C_{IJ},C_{IK}$  to some \textit{canonical bases}   $(\{ X_i \} , \{ Y_i \} ,\{ Z_i \}) \in \mathcal{B}(U)$ built up through the determination of the mutual principal angles and associated principal vectors  of the  uniquely defined \textit{standard subspaces}  $U_i^I,U_j^J,U_k^K$ of the forms $\omega^I|_U,\omega^J|_U,\omega^K|_U$ respectively.

 In  \cite{Vacpreprint} we showed in fact that a necessary condition for a pair of subspaces $U,W$ to belong to the same $Sp(n)$-orbit is that the principal angles between the  $\omega^I,\omega^J,\omega^K$ standard subspaces $(U_i^I,U_j^J)$ as well as  $(U_i^I,U_k^K)$  and $(U_j^J,U_k^K)$ are the same, where the indices $i,j,k$ range through all standard subspaces.

 But such procedure requires that at each step there exists at least one pair of standard  subspaces (whose size  reduce at the end of the current  iteration) with one or more principal angles of multiplicity equal to one.
 In most of the cases this procedure leads to the construction of the canonical matrices  associated to $U$ and consequently to the full set of invariants (for a fixed admissible basis)  characterizing  its $Sp(n)$-orbit.  But there are some \textbf{degenerate cases} (see \cite{Vacpreprint}) where the aforementioned procedure  can not be applied tout court.
 Among these, there are the \textit{isoclinic} subspaces i.e. subspaces $U$ characterized by the fact that,   for all  $A \in S(\mathcal{Q})$,  all principal angles between $U$ and $AU$ equal $\theta^A$  which is  called  the angle  of isoclinicity.  They do not exhaust all the degenerate cases (see  \cite{Vacpreprint}) but surely   they have an important role and this article deals with them.

 The purpose of this article is to determine the $Sp(n)$-orbit in the real Grassmannian $G^\R(k,4n)$ of $k$-dimensional  isoclinic subspaces and we will end up with a theorem giving the full set of invariants characterizing their orbits.
We denote by $\mathcal{IC}^{2m}$ the set of $2m$-dimensional isoclinic subspaces. 
We consider only even dimensions since the set $\mathcal{IC}^{2m+1}$ is formed only by  the real Hermitian product  (r.h.p.)  subspaces  $U$  (see \cite{articolo_2}) characterized by the fact that  $U \perp AU, \; \forall  A \in S(\mathcal{Q})$. From Theorem (\ref{main_theorem 1_ existance of canonical bases with same mutual position}) one has that all and only   such isoclinic subspaces share the same orbit in $G^\R(2m+1,4n)$.

 First  we prove that, fixed  an admissible basis $(I,J,K)$, for any unitary $X_1 \in U$, and being $X_2,Y_2,Z_2$ the triple of unitary vectors orthogonal to $X_1$  such that the 2-planes $L(X_1,X_2),  L(X_1,Y_2),L(X_1,Z_2)$ are standard 2-planes of $\omega^I|_U, \omega^J|_U, \omega^K|_U$ respectively, the cosines $\xi=<X_2,Y_2>, \; \chi=<X_2,Z_2>, \; \eta=<Y_2,Z_2>$  do not depend on $X_1$.

Furthermore, we prove that, fixed an admissible basis $(I,J,K)$,  any $U \in \mathcal{IC}^{2m}, \; m>1$ is characterized by  a pair of invariants $(\Gamma,\Delta)$  which, together with $(\xi,\chi)$, determine the  canonical form of the matrices $C_{IJ},C_{IK}$.  Being $\Gamma$ a function of $(\xi,\chi,\eta)$, according to the Theorem (\ref{main_theorem 1_ existance of canonical bases with same mutual position}),  one has that $(\xi,\chi,\eta,\Delta)$ and the triple $(\theta^I,\theta^J,\theta^K)$ of the angles of isoclinicity  determine the $Sp(n)$-orbit of $U$.

We say that a pair of isoclinic subspaces $U,W$,  even of different dimensions, have the same angles of isoclinicity when, for   all compatible complex structure  $A\in S(\mathcal{Q})$, the angles of isoclinicity  of the  pairs $(U,AU)$ and $(W,AW)$  are the same. We also say that $U$ is a 2-planes decomposable subspace if it admits an orthogonal decomposition into isoclinic 2-planes with same angles of isoclinicity of $U$. Furthermore we say that $U$ is orthogonal if at least one among $(\theta^I, \theta^J,\theta^K)$ equals $\pi/2$.
 The main result of this paper is given in the Theorem (\ref{Orbit of an isoclinic subspace}) that here we report:
\begin{teor}
Let  $U \in \mathcal{IC}^{2m}$. Let fix an admissible basis $(I,J,K)$  and denote by  $(\theta^I,\theta^J,\theta^K)$ the  angles of isoclinicity of the pairs $(U,IU),(U,JU),(U,KU)$ respectively. For $k \geq 0$:
\begin{itemize}
\item If $2m= 8k+2$ or  $2m= 8k+6$, $U$ is 2-planes decomposable i.e. is orthogonal sum of $U_i \in  \mathcal{IC}^2$ with same angle of isoclinicity of $U$.
In this case $(\Gamma,\Delta)=(1,0)$ and  the pair  $(\xi,\chi)=(\pm 1, \pm 1)$   determine the matrices $C_{IJ},C_{IK}$.
The $Sp(n)$-orbit is then determined by the angles $(\theta^I,\theta^J,\theta^K)$ and by the pair $(\xi,\chi)$.
\item If $2m= 8k+4$, then $U$  is orthogonal sum of $U_i \in  \mathcal{IC}^4$ with same angle of isoclinicity of $U$ and characterized by the same pair $(\Gamma,\Delta)$. In this case $\Gamma^2+\Delta^2=1$ and the canonical matrices are determined by $(\xi,\chi,\Gamma,\Delta)$.  In particular this case always occurs if $U$ is orthogonal in which case $(\Gamma,\Delta)=(1,0)$.
    The $Sp(n)$-orbit is then characterized by   $(\theta^I,\theta^J,\theta^K)$ and $(\xi,\chi,\eta,\Delta)$.
    In particular, if $\xi=\pm 1$ and $\chi=\pm 1$ we are  in the first case.
\item If $2m= 8k$ then   $U$  is orthogonal sum of $U_i \in  \mathcal{IC}^8$ with same angle of isoclinicity of $U$. The canonical matrices are determined by $(\xi,\chi,\Gamma,\Delta)$  where $\Gamma^2 + \Delta^2 \leq 1$ and the $Sp(n)$-orbit by $(\theta^I,\theta^J,\theta^K)$ and $(\xi,\chi,\eta,\Delta)$. If in particular $\Gamma^2 + \Delta^2=1$ we are in the previous case and if
     furthermore  $\xi=\pm 1$ and $\chi=\pm 1$ we are  in the first case.
\end{itemize}
\end{teor}

\section{Preliminaries}
In this paragraph  we recall some definitions regarding the structures, angles and groups we need in this paper.  For a wider treatment one can refer to \cite{Vac}.
Let $V$ be a $4n$-dimensional real vector space. We endow $V$ with an Hermitian quaternionic structure    and an ($\mathbb{H}$-valued)-Hermitian product (see \cite{Vac})  whose definitions  we recall in the sequel.

\begin{defi}
\begin{enumerate}
\item A triple $\mathcal{H}=\{J_1,J_2, J_3 \}$ of anticommuting complex structures  on $V$ with $J_1J_2= J_3$ is called a \textbf{hypercomplex structure} on $V$.
\item The 3-dimensional subalgebra
\[ \mathcal{Q}= span_\R(\mathcal{H})= \R J_1 + \R J_2 + \R J_3 \approx \mathfrak{sp}_1 \]
of the Lie algebra $End(V)$ is called a \textbf{quaternionic structure} on $V$.
\end{enumerate}
\end{defi}
Note that two hypercomplex structures $\mathcal{H}= \{J_1,J_2,J_3\}$ and $\mathcal{H'}= \{J'_1,J'_2,J'_3\}$ generate the same quaternionic structure $\mathcal{Q}$ iff they are related by a rotation, i.e.
\[
J'_\alpha= \sum_\beta A^\beta_\alpha J_\beta,\qquad (\alpha= 1,2,3)
\]
with $(A_\alpha^\beta) \in SO(3)$. A  hypercomplex structure  generating  $\mathcal{Q}$ is called an \textbf{admissible} (hypercomplex) \textbf{basis} of $\mathcal{Q}$. We denote by $S(\mathcal{Q})$ the 2-sphere of complex structures in $\mathcal{Q}$ i.e. $S(\mathcal{Q})= \{aJ_1 + bJ_2 + cJ_3,\;  a,b,c \in \R, \; a^2 + b^2 + c^2=1 \}$ .

\begin{defi}
An Euclidean scalar product $< \, , \,>$ in $V$ is called \textbf{Hermitian} with respect to  a hypercomplex structure  $\mathcal{H}= (J_\alpha)$ (resp. the quaternionic structure $\mathcal{Q}= span_\R(\mathcal{H})$) if and only if,  for any $X,Y \in V$,
\[
<J_\alpha X, J_\alpha Y>=  < X,  Y>, \qquad (\alpha=1,2,3)
\]
(respectively
\[
<J X, J Y>=  < X,  Y>, \qquad (\forall J \in S(\mathcal{Q})) \text{)}.\]
\end{defi}

\begin{defi}
  A hypercomplex structure  $\mathcal{H}$ (resp. quaternionic structure $\mathcal{Q}$) together with an Hermitian scalar product $< \, , \,>$ is called an \textbf{Hermitian hypercomplex} (resp. \textbf{Hermitian quaternionic}) \textbf{structure} on $V$  and the triple $(V^{4n}, \mathcal{H},<,>)$ (resp. $(V^{4n}, \mathcal{Q}, <, >)$) is  an \textbf{Hermitian  hypercomplex} (resp. \textbf{  quaternionic} ) \textbf{ vector space}.
\end{defi}

For an introduction and a  survey of some results on Hermitian hypercomplex and Hermitian quaternionic structures  one can refer among others to  \cite{Vac,AM,Sa}.

The group $Sp(1)$
is  the group under multiplication of unitary quaternions.
It is a Lie group whose  Lie algebra $\mathfrak{sp}_1= Im \; \mathbb{H} \simeq \mathcal{Q}$.
For any  quaternion $q \in Sp(1)$, let consider the unitary homothety in the $\mathbb{H}$-module $V$.
\[ q: X \mapsto Xq, \quad X \in V.\]

To these transformations belong for instance the automorphisms $I=R_{-i},J=R_{-j},K=R_{-k}$ given by the right multiplications by $-i, -j, -k$  being $(1,i,j,k)$  a basis of  $\mathbb{H}$ satisfying the multiplication table obtainable from the conditions
\begin{equation} \label{multiplication table of quaternions}
i^2=j^2=k^2=-1; \; ij=-ji=k.
\end{equation}
\begin{prop} \cite{BR2}
We denote by $\mathcal{B}$ the set of such bases. The unitary homotheties are rotations of $V^{4n}$ that leave invariant any characteristic line.
\footnote{The following definition appeared in \cite{BR1}. We regard $V \cong \R^{4n}$  as a right module over the skew-field $\mathbb{H}$ of quaternions by
identifying $\R^{4n}$ with $\mathbb{H}^n$ and by letting $\mathbb{H}$ act by right multiplication.
The subspaces $U^{4h}\subset \mathbb{H}^n$ of real dimension $4h$ real image of the subspaces of  $\mathbb{H}^n$ of quaternionic dimension $h$ are called \textbf{characteristic  subspaces}. A characteristic subspace of dimension 4 (resp. $8,12 \ldots$) is called characteristic line,  (resp. plane, 3-plane, $\ldots$). According to the definition given in \cite{articolo_2} a characteristic line is a 4-dimensional quaternionic  subspace.}

Moreover for any $X \in V$ the angle $\widehat{X,Xq}$ does not depend on $X$ and it is
\[
\cos \widehat{X,Xq}= Re(q)\]
\end{prop}

Restricting
to the action of $Sp(1)$ determines then an inclusion
\[ \lambda: Sp(1)  \hookrightarrow SO(4n).\]

We define $Sp(n)$ to be the subgroup of $SO(4n)$ commuting with $\lambda(Sp(1))$ i.e. $Sp(n)$ is the centralizer of $\lambda Sp(1)$ in $SO(4n)$.
Then the group $Sp(n)$ is the group of automorphisms of an Hermitian hypercomplex vector space.

As  an  $\mathbb{H}$-module,  on a  quaternionic
Hermitian vector space $(V^{4n},\{I,J,K\}_\R,<,>)$, once identified the hypercomplex basis with $(R_{-i},R_{-j},R_{-k})$ for some basis $(i,j,k)$  of $Im(\HH)$ with $(1,i,j,k) \in \mathcal{B}$,  we define the
($\mathbb{H}$-valued)-Hermitian  product $(\cdot)$  by:
\begin{equation} \label{Hermitian product in an Hermitian quaternionic vector space}
\begin{array}{lllll}
( \cdot ) : & V \times V & \rightarrow &
 \HH \\
 & (X,Y) & \mapsto &  X \cdot Y &= <X,Y> +  <X,IY>i +  <X,JY>j +  <X,KY>k. \\
 &  &  & &  = <X,Y> +  \omega^I(X,Y)i +  \omega^J(X,Y)j +  \omega^K(X,Y)k.
\end{array}
\end{equation}

The Hermitian product in (\ref{Hermitian product in an Hermitian quaternionic vector space}) is
 definite positive.  Observe that if the pair $(X,Y)$ is an orthonormal oriented basis of the   2-plane $U=L(X,Y)$,  for any $A \in S(\mathcal{Q})$,  $\omega^A(X,Y) \in [-1,1]$ is the cosine of the $A$-K\"{a}hler angle $\Theta^I(U)$ (see definition at the end of this section).

The Hermitian product in $\HH^n$    is canonical  (\cite{BR1, Vac}) i.e. it does not depend on  the particular basis $(1,i,j,k) \in \mathcal{B}$.
Observe that, when considering  a different admissible basis $(I',J',K')=(I,J,K) C, \, C \in SO(3)$, the new  quantities $<X,I'Y>, <X,J'Y>, <X,K'Y>$  in (\ref{Hermitian product in an Hermitian quaternionic vector space})  are now the component of $X \cdot Y$ w.r.t. the basis $(i',j',k')=(i,j,k)C$ of $Im(\HH)$.    If instead one  considers a fixed basis $B \in \mathcal{B}$, the Hermitian  product (\ref{Hermitian product in an Hermitian quaternionic vector space}) is defined up to an inner automorphisms of $\mathbb{H}$.

 We recall  the definition of the principal angles between a pair of subspaces of a real vector space $V$ (see \cite{Gala},\cite{Riz1} among others).
\begin{defi} \label{definition of principal angles}  Let $A,B \subseteq V$ be subspaces, \mbox{$\dim k= dim(A) \leq dim(B)=l \geq 1$.} The \textbf{principal angles} $\theta_i \in [0,\pi/2]$ between the subspaces $A$ and $B$ are recursively defined for $i= 1, \ldots,k$   by
\[
\cos \theta_i= \frac{<a_i,b_i>}{||a_i|| \, ||b_i||}= \max \{ \frac{<a,b>}{||a|| \, ||b||} \, : a \perp a_m, \, b \perp b_m, \,  m= 1,2, \ldots, i-1 \}
\]
where the $a_j \in A, \, b_j \in B$. The pairs $(a_i,b_i),\;i=1, \ldots, k$ are called  \textbf{related principal vectors}. \\
\end{defi}
In words, the procedure is to find the unit vector $a_1 \in A$ and  the unit vector $b_1 \in B$ which minimize the angle between them and call this angle $\theta_1$.
Then consider the orthogonal complement in $A$ to $a_1$   and the orthogonal complement in $B$ to $b_1$ and iterate.
The principal angles $ \theta_1, \ldots, \theta_k$ between the pair of subspaces $A,B$ are  some of the critical values of the angular function
\[ \phi_{A,B}= A \times B \rightarrow \R\]
associating with each pair of non-zero vectors $a \in A, \; b \in B$ the angle between them. 
Moreover the principal angles are the diagonal entries of the orthogonal projector $P^A: B \rightarrow A$  stated in the theorem  of Afriat (\cite{GS}, \cite{Af}):
\begin{teor} \label{Afriat Theorem}\cite{GS}, \cite{Af}
In any pair of subspaces $A^{k}$ and $B^{l}$ there exist orthonormal bases $\{u_i\}_{i=1}^{k}$ and $\{v_j\}_{j=1}^{l}$
such that $<u_i , v_i> \geq 0$ and $<v_i \times v_j>=0$ if $i \neq j.$
\end{teor}
\begin{proof}
It is a direct consequence of the following
\begin{lemma} \label{principal angles are singular values of projector}
Given finite dimensional subspaces $A,B$, let $a_1,b_1$ attain
\[
\max \{<a,b>, \quad a \in A,  \; b \in B, \quad ||a||=1, \; ||b||=1\}
\]
(i.e. the pair $(a_1,b_1)$ are the first principal vectors). Then
\begin{enumerate}
\item for some $\alpha \geq 0$,
\[P^B a_1= \alpha b_1, \qquad P^A b_1=\alpha a_1\]
\item  $a_1  \perp (b_1^\perp \cap B)$ and  $b_1  \perp (a_1^\perp \cap A)$ which leads the diagonal form of the matrix of the Projector $P^B$ (and $P^A$).
\end{enumerate}
\end{lemma}
To see that 1) holds, note that $P^B a_1= \alpha b$ where $\alpha, b$ minimize $||a_1-\alpha b||^2$ for $b \in B, ||b||=1$  and $\alpha$ a scalar. Thus to  minimize   $||a_1-\alpha b||^2= \alpha^2 -2 \alpha <a_1,b> +1$  we must maximize $<a_1,b>$. Moreover $\alpha=<a_1,b_1>$ is the cosine of the first principal angle.

For 2), let $A_1=a_1^\perp \cap A$ (resp. $B_1=b_1^\perp \cap B$). If $a \in A_1$, then $a \perp b_1$ since $ <a,b_1>= <P^Aa,b_1>=<a, P^A b_1>=<a,\alpha a_1>= 0$.  Likewise if $b \in B_1$ then $b \perp a_1$. We proceed letting $a_2$ and $b_2$ attain
\[
\max \{<a,b>, \quad a \in A_1,  \; b \in B_1, \quad ||a||=1, \; ||b||=1\}
\]
and continue till we have exhausted $A$ and $B$.
\end{proof}

From the proof of  (\ref{Afriat Theorem}) we have that the principal angles between a pair of subspaces $A,B$ of $V$ can
 also defined as the singular value of the orthogonal projector $P^A$ (or equivalently  $P^B$)

Let  recall the definition and some properties of isoclinic subspaces.

\begin{defi} \label{definition of isoclinicity}
A pair of non oriented subspaces $A$ and $B$ of same dimension are said to be \textbf{isoclinic} and the angle $\phi$  ($0 \leq \phi \leq \frac{\pi}{2}$)
is said to be angle of isoclinicity  between them  if either of the following conditions hold:\\
1) the angle between any non-zero vector of one of the subspaces and the other subspace is equal to $\phi$;\\
2) $G G^t= \cos^2 \phi \, Id$ for the matrix $G=<a_i,b_j>$ of the orthogonal projector $P_B^A: B \rightarrow A$  with respect to  any orthonormal basis $\{a_i\}$ of $A$ and $\{b_j\}$ of $B$;\\
3) all principal angles between $A$ and $B$ equal $\phi$.
\end{defi}

\begin{defi}
We denote by $\mathcal{IC}^{2m}$ the set of $2m$-dimensional subspaces of $V$ such that,  for any $A \in S(\mathcal{Q})$,   the pair $(U,AU)$ is isoclinic. When we do not need to specify the dimension we just use the notation $\mathcal{IC}$ and we call them simply \textbf{isoclinic subspaces}.
\end{defi}
 The fact that we consider only  even dimensions subspaces follows from the
 \begin{prop}
 Let $U$ be an  odd dimension isoclinic subspace. Then $U$  is a real hermitian product subspace (r.h.p.s.) i.e. for one and hence any admissible basis $(I,J,K)$ the pairs $(U,IU),(U,JU),(U,KU)$ are strictly orthogonal. Then $\mathcal{IC}^{2m+1}$ is the set of all and only the real Hermitian product $(2m+1)$-dimension subspaces.
  All  r.h.p. subspaces   share the same orbit.
 \end{prop}
   \begin{proof}
   The first statement is obvious since,  for any $A \in S(\mathcal{Q})$, by the skew-symmetry of $\omega^A$ one (and then all) principal angle is necessarily equal to $\pi/2$.
   For the last statement  observe that  any orthonormal basis $B$ is a standard basis of $\omega^I,\omega^J,\omega^K$  restricted to $U$ which implies that w.r.t $B$  one has $C_{IJ}=C_{IK}=Id$  and the conclusion follows from the  aforementioned Theorem (\ref{main_theorem 1_ existance of canonical bases with same mutual position}).
   \end{proof}
   Therefore, this paper  deals with  isoclinic subspaces of  even dimension.  Throughout the article we will fix an  admissible basis $(I,J,K)$  and, given $U \in \mathcal{IC}^{2m}$, we denote by $\theta^I, \theta^J, \theta^K$ the respective angles of isoclinicity.  If  the pair $(U,IU)$  (resp. $(U,JU)$, resp. $(U,KU)$)  is strictly orthogonal (i.e. if all principal angles are $\pi/2$)
 we say that $U$ is $I$-orthogonal (resp. $J$-orthogonal, resp. $K$-orthogonal) and in general we speak of single orthogonality (or 1-orthogonality).  When  two  (resp. three) of the above pair are strictly orthogonal we speak of double  (resp. triple)-orthogonality. By saying that $U$ is \textbf{orthogonal} (without specifying the complex structures) we mean that at least one  among $\theta^I,\theta^J,\theta^K$ equals $\pi/2$.

Fixed an admissible basis $\mathcal{H}$ of $\mathcal{Q}$, throughout this paper we will define some functions  $f: V \times V \times \ldots \times V \rightarrow \R$. If $A$ is one of them and it is  constant  on its domain,  will say that $A$ is an  \textit{invariant}  of $U$. If furthermore  the invariant $A$ does not depend on the chosen hypercomplex basis $\mathcal{H}$, we will say that $A$ is  an \textit{intrinsic property}  of $U$.

 Given a pair of non oriented  subspaces $(U,W)$ we denote by $\widehat{U,W}$ the Euclidean  angle $\phi \in [0,\pi/2]$ they form.
   We recall (\cite{SC}) that if  $\theta_1, \ldots, \theta_p$  are the principal angles between $U$ and $W$ one has
   \[
\cos \phi= \cos \theta_1 \cdot \cos \theta_2 \cdot  \ldots  \cdot \cos \theta_p.
\]

 Finally we recall the notion of  K\"{a}hler angle which is defined in a real  vector space $V$ endowed with a complex structure $I$.
 \begin{defi} Let $(V^{2n},I)$ be a real vector space endowed with a complex structure $I$.
 For any pairs of non parallel vectors  $X,Y \in V$ their  \textbf{K\"{a}hler angle} is given by
 \begin{equation} \label{definition of Kaehler angle}
 \Theta^I=\arccos \frac{<X,IY>}{|X| \, |Y| \sin \widehat{XY}}=\arccos \frac{<X,IY>}{mis \; (X \wedge Y)}.
 \end{equation}
  \end{defi}
Then $0 \leq \Theta^I \leq \pi$.
 It is straightforward to check that the K\"{a}hler angle is an intrinsic property of the oriented 2-plane $U=L(X,Y)$.  For this reason one  speaks of the K\"{a}hler angle  of an oriented  2-plane.

The cosine of the K\"{a}hler angle of the pair of 2-planes  with opposite orientation $U$ and $\tilde U=L(Y,X)$  have opposite sign i.e. $\cos \Theta^I(U)=-\cos \Theta^I(\tilde U)$,
then, if one disregards the orientation of the 2-plane $U$, we can consider the absolute value of the right hand side of equation (\ref{definition of Kaehler angle})   restricting the K\"{a}hler angle to the interval $[0,\pi/2]$. In this case   the K\"{a}hler angle of the 2-plane $U$ coincides with one of the two identical principal angles, say $\theta^I(U)$, between the pairs of 2-plane $U$ and $IU$ (same as the pair  $(\tilde U,I \tilde U)$) which  are always isoclinic as one can immediately verify, then
\[\cos \theta^I(U)=|\cos \Theta^I(U)|\]
and  one has
 \begin{equation} \label{angle between 2-planes U,IU}
   \cos (\widehat{U,IU})= \cos (\widehat{\tilde U,I \tilde U})=  \frac{<X,IY>^2}{mis^2 \; (X \wedge Y)}= \cos^2 \theta^I(U) =  \cos^2 \Theta^I(U).
\end{equation}
The K\"{a}hler angle measures the deviation of a 2-plane from holomorphicity. For instance the K\"{a}hler angle of a totally real plane $U$ (i.e. a plane such that $U \perp IU$) is  $\Theta^I(U)=\pi/2$  and  the one of an oriented  complex plane $U$ (i.e. $U=IU$) is  $\Theta^I(U) \in \{0, \pi\}$.

Generalizing  the notion of K\"{a}hler angle,   in an Hermitian quaternionic vector space  $(V^{4n}, \mathcal{Q}, <,>)$ we will  speak of the  \textbf{$A$-K\"{a}hler angle}  of  an oriented  2-plane $U$ with   $A \in S(\mathcal{Q})$. We will denote  it by $ \Theta^A(U)$.


\section{2-dimensional subspaces}
The simplest example of an even dimension isoclinic subspace $U \subseteq V$ is a 2-plane. By the skew-symmetry of the $A$-K\"{a}hler form for any $A \in  S(\mathcal{Q})$, any  2-dimensional subspace of $U$ is isoclinic with $AU$.  Therefore  as a set one has that $G_\R(2,4n)= \mathcal{IC}^2$.  The study of the orbits of the oriented 2-planes in the real Grassmannian 
  under the action of the groups $Sp(n)$ and $Sp(n) \cdot Sp(1)$ is carried out in   \cite{Vac}.
There we introduced the \textit{imaginary measure} and the \textit{characteristic deviation} of an oriented  2-plane $U$   proving that such invariants characterize its  orbits under the action of the groups $Sp(n)$ and $Sp(n) \cdot Sp(1)$.
 Namely, let $L,M$ be an oriented basis of $U$. The  purely imaginary quaternion
\begin{equation} \label{definition of imaginary measure}
 \mathcal{IM}(U)=
 \frac{Im (L \cdot M)}{mis(L \wedge M)}, 
 \end{equation}
  is  an intrinsic property of  an oriented 2-plane \mbox{$U \subset (V^{4n}, \mathcal{Q}, <,>)$}  i.e. it does not depend neither on  the chosen oriented generators $L,M$ nor on the admissible  basis $\mathcal{H}$ of $\mathcal{Q}$.  Moreover $Sp(n)$ preserves $\mathcal{IM}(U)$.
In particular, if the pair  $L,M$ is an orthonormal oriented  basis of $U$, then $\mathcal{IM}(U)=L \cdot M$.
   We called it \textit{imaginary measure} of the oriented 2-plane $U$.
     Disregarding the orientation of  $U$  and being  $(L,M)$  some orthonormal basis,  it is $\mathcal{IM}(U)= \{\pm \, L \cdot M\}$ i.e.  is the set made of a pair of conjugated pure imaginary quaternions.
  We  proved   that
   \begin{teor} \cite{Vac}  \label{Sp(n)-orbit of a 2-plane}
   The imaginary measure $\mathcal{IM}(U)$ represents  the full system of invariants for the $Sp(n)$-orbits in the real Grassmannian of 2-planes $G_\R(2,4n)$ as well as in $G_\R^+(2,4n)$ (the Grassmannian  of the oriented  2-planes) i.e. a pair of 2-planes $(U,W)$ of $(V^{4n}, \mathcal{Q}, <,>)$ are in the same $Sp(n)$-orbit iff $\mathcal{IM}(U)=\mathcal{IM}(W)$.
 \end{teor}

Let   consider   a triple of  standard bases $(X_1,X_2), (X_1,Y_2),(X_1,Z_2)$ with a common leading vector $X_1$ of the non oriented 2-plane $U$. 
 By definition one has that  $\cos \theta^I=<X_1,IX_2>,\cos \theta^J=<X_1,JY_2>,\cos \theta^K=<X_1,KZ_2>$ are non negative, and computed $\xi=<X_2,Y_2>,\chi=<X_2,Z_2>, \eta=<Y_2,Z_2>$, where $(\xi,\chi,\eta) \in \{-1,1\}$ one has that
the matrices $C_{IJ}$ and  $C_{IK}$ w.r.t. the standard bases $(X_1,X_2),(X_1,Y_2)(X_1,Z_2)$  are given by
\[
C_{IJ}:
 \left(   \begin{array}{cc}
1 & 0 \\
0 & \xi
\end{array}      \right) \qquad
C_{IK}:
 \left(   \begin{array}{cc}
1 & 0 \\
0 & \chi
\end{array}
\right)
\]
It is straightforward to verify that  the pair $(\xi,\chi)$ is an invariant of $U$.
 Therefore, according to Theorem (\ref{main_theorem 1_ existance of canonical bases with same mutual position}), together with the triple $(\theta^I,\theta^J,\theta^K)$, such pair
 determines  the $Sp(n)$-orbits of the (non oriented) 2-plane $U$.

This is accordance with the Theorem (\ref{Sp(n)-orbit of a 2-plane}). In fact,   If $U$ has a triple orthogonality then  clearly $\mathcal{IM}(U)=0$. In this case any orthonormal basis is a standard basis of $\omega^I,\omega^J,\omega^K$ which implies $\xi=\chi=1$.   Else  suppose, without lack of generality, that $\cos \theta^I \neq 0$ and  that let
$(X_1,X_2)$ be an $\omega^I$-standard basis. Then
\[\mathcal{IM}(U)= X_1 \cdot X_2= \pm( \cos \theta^I i + \xi \cos \theta^J j + \chi \cos  \theta^K k).\]

 Given a pair  of 2-planes  $U,W$ with $\mathcal{IM}(U)=\mathcal{IM}(W)$,  according to Theorem (\ref{Sp(n)-orbit of a 2-plane}), they are in the same orbit. Since   they share the same  pair $(\xi,\eta)$ and the same triple $(\theta^I,\theta^J,\theta^K)$,   they are in the same $Sp(n)$-orbit also according to
Theorem (\ref{main_theorem 1_ existance of canonical bases with same mutual position}). Viceversa  if they share the same  pair $(\xi,\eta)$ and the same triple $(\theta^I,\theta^J,\theta^K)$  which implies that they belong to the same $Sp(n)$-orbit according to  Theorem (\ref{main_theorem 1_ existance of canonical bases with same mutual position}) then clearly  $\mathcal{IM}(U)=\mathcal{IM}(W)$.

\section{4-dimensional isoclinic subspaces}
\subsection{Invariants of a 4 dimensional  isoclinic subspace}

Let consider a 4-dimensional subspace $U$  such that the pair $(U,AU), \; A \in S(\mathcal{Q})$  is  isoclinic with angle $\theta^A$.
From 2) of the Definition (\ref{definition of isoclinicity}), for any (principal) vector $X$ in $U$, one has $Pr_{AU}^{U} \circ Pr_U^{AU}(X)= \cos^2 \theta^A X$.
 Any pair of orthonormal basis $(X_1,X_2,X_3,X_4)$ of $U$  and  $(AX_1,AX_2,AX_3,AX_4)$  of $AU$  are made of principal vectors of $U$ and $AU$ respectively. In general the pairs $(X_i,AX_i)$ are not \textit{related} principal vectors unless $U \perp AU$.

 Then, w.r.t. the  pair of  aforementioned bases, for the skew-symmetric matrix of the projector $Pr_U^{AU} : U \rightarrow AU$  (or equivalently for the  matrix of the form $\omega^A$ w.r.t. the basis $(X_1,X_2,X_3,X_4)$ of $U$)  the condition $(\omega^A)^t \circ \omega^A=\cos^2 \theta^A \, Id$ leads to the following

\begin{prop} \label{general form for the matrix of projector omega^A of an isoclinic 4 dimensional subspace}
Let $A \in S(\mathcal{Q})$. The pair $(U,AU)$ of 4-dimensional subspaces is isoclinic iff the matrix of $\omega^A$ w.r.t. the orthonormal basis $(X_1,X_2,X_3,X_4)$  has the form
\begin{equation}\label{matrix of projector omega^A of an isoclinic 4 dimensional subspace}
\omega^A:
 \left(   \begin{array}{cccc}
0 & a & b & c\\
-a & 0 & \pm c & \mp b\\
-b & \mp  c & 0 & \pm a\\
-c & \pm b & \mp a & 0
\end{array}      \right).
\end{equation}
\end{prop}
It is a matrix with orthogonal rows and columns  (see \cite{YIU}) 
whose square norms  evidently  equal the square cosine of  the angle of isoclinicity  $\theta^A$ between the pair $(U,AU)$ i.e. 
\[
\cos \theta^A= \sqrt{a^2 + b^2 + c^2}.
\]
Moreover, recalling that for any orthonormal basis $(X_1,X_2,X_3,X_4)$ $a,b,c$ represent  respectively
the cosines of the $A$-K\"{a}hler angles  $\cos \Theta^A(U_{12}), \, \cos \Theta^A(U_{13}), \, \cos \Theta^A(U_{14})$  where
 $U_{1j}=L(X_1,X_j), \, j=2,3,4$ and from (\ref{angle between 2-planes U,IU})  one has
  \[ \cos^2 \theta^A= \cos^2 \Theta^A(U_{12})^2 + \cos^2 \Theta^A(U_{13})^2 + \cos^2 \Theta^A(U_{14})^2 =\cos (\widehat{U_{12},AU_{12}}) +\cos (\widehat{U_{13},AU_{13}}) + \cos (\widehat{U_{14},AU_{14}}) .\]

Observe that, from  the above definition, given a pair of vectors $X,Y$ and being the pair of  4-dimensional quaternionic subspaces $\mathcal{Q}X=L(X,IX,JX,KX)$ and $\mathcal{Q}Y$  always isoclinic as it can be easily verified, the Hermitian and the characteristic angle  of any pair of vectors  $X,Y$  (see \cite{Vac})
\footnote{The \textbf{quaternionic characteristic angle} $\varphi$ between  a pair of vectors $L,M$ of a quaternionic vector space $V^{4n}$  is given by
\begin{equation} \label{Quaternionic characteristic angle}
\cos \varphi  = \frac{[ \mathcal{N}(L \cdot  M)]^2}{mis^4 \; L \; mis^4 \; M}= \frac{(<L,M>^2 +<L,IM>^2+<L,JM>^2+<L,KM>^2)^2}{<L,L>^2 <M,M>^2}.
\end{equation}
The \textbf{Hermitian angle} between the same  pair of vectors of  $V$  is defined as
\begin{equation}
\cos \psi= \frac{ |(L \cdot  M)|}{|L| |M|}= \frac{\sqrt{(<L,M>^2 +<L,IM>^2+<L,JM>^2+<L,KM>^2)}}{\sqrt{<L,L>} \sqrt{<M,M>}}.
\end{equation}
Both angles do not depend on the admissible basis of $\mathcal{Q}$. Therefore the Hermitian angle $\psi$ between a pair of vectors $L,M$ is just the angle between such pair computed by using  the Hermitian product, whereas the characteristic angle $\varphi$ is the  angle between the   4-dimensional characteristic lines  they span over $\mathbb{H}$. It is $\cos \varphi=\cos^4 \psi$.}
equal respectively the angle of isoclinicity of the pair $(\mathcal{Q}X,\mathcal{Q}Y)$
 and the (Euclidean) angle between the same pair of subspaces.

Consider now  a subspace $U$ such that the pairs $(U,IU)$, $(U,JU)$, $(U,KU)$ are isoclinic where $(I,J,K)$ is some admissible basis. To this set belong for instance
all totally complex, quaternionic and real Hermitian product (r.h.p.)  subspaces (see \cite{articolo_2}).
In particular in this paragraph we  consider a 4-dimensional subspace $U$.   We can add to the  list above all 4-dimensional complex subspaces (see \cite{articolo_2}).
Let fix an orthonormal basis $B=(X_1,X_2,X_3,X_4)$ of $U$. Keeping the same notations used in the Proposition (\ref{general form for the matrix of projector omega^A of an isoclinic 4 dimensional subspace}), we denote by  $(a,b,c), (a',b',c'), (\tilde a, \tilde b, \tilde c)$ the  entries of the first row of the matrices representing respectively $\omega^I, \omega^J,\omega^K$ w.r.t. the basis $B$   whose form is given in (\ref{matrix of projector omega^A of an isoclinic 4 dimensional subspace}).

\begin{prop} \label{isoclinicity w.r.t. one hypercomplex basis implies isoclinicity w.r.t. any compatible complex structure}
Let $U$ be  a 4 dimensional subspace and $(I,J,K)$ an admissible basis. Suppose the pairs $(U,IU)$, $(U,JU)$, $(U,KU)$ are isoclinic and $\theta^I, \theta^J, \theta^K$ the respective angles of isoclinicity. Then for any  $A= \alpha_1 I + \alpha_2 J + \alpha_3 K \in S\mathcal(Q)$  the pair $(U,AU)$ is isoclinic and therefore $U \in \mathcal{IC}^4$.
The angle of isoclinicity $\theta^A$ between the pair $(U,AU)$  is given by
\begin{equation} \label{angle of isoclinicity of a 4 dimensional subspace isoclinic w.r.t. an hypercomplex basis 1}
\cos^2 \theta^A =-\frac{1}{4} Tr [(\alpha_1 \omega^I + \alpha_2 \omega^J + \alpha_3 \omega^K)^2]= -\frac{1}{4} Tr [(\omega^A)^2]
\end{equation}
\end{prop}

\begin{proof}
The first statement follows from the Proposition (\ref{general form for the matrix of projector omega^A of an isoclinic 4 dimensional subspace}) since $\omega^A= \alpha_1 \omega^I + \alpha_2 \omega^J + \alpha_3 \omega^K$  has clearly the form given in (\ref{general form for the matrix of projector omega^A of an isoclinic 4 dimensional subspace}).
For the angle of isoclinicity and considering for instance the vector $X_1$ of $B$,  one has \\ $\cos^2\theta^A = <X_1,AX_2>^2 + <X_1,AX_3>^2 + <X_1,AX_4>^2$  and computing we have
  \begin{equation} \label{angle of isoclinicity of a 4 dimensional subspace isoclinic w.r.t. an hypercomplex basis 2}
\cos^2 \theta^A =\alpha_1^2 \cos^2 \theta^I + \alpha_2^2\cos^2 \theta^J + \alpha_3^2\cos^2 \theta^K+ 2 \alpha_1 \alpha_2(a a' + bb' + cc') + 2 \alpha_1 \alpha_3(a \tilde a + b \tilde b + c \tilde c)+ 2 \alpha_2
\alpha_3(a' \tilde a + b' \tilde b + c' \tilde c).
\end{equation}
It is
$(a a' + bb' + cc')= -\frac{1}{4} Tr (\omega^I \cdot \omega^J)$  (resp. $(a \tilde a + b \tilde b + c \tilde c)=-\frac{1}{4} Tr (\omega^I \cdot \omega^K)$, resp. $(a' \tilde a + b' \tilde b + c' \tilde c)=-\frac{1}{4} Tr (\omega^J \cdot \omega^K)$). Moreover
$Tr [(\omega^I)^2]= -4 \cos^2 \theta^I, \; Tr [(\omega^J)^2]= -4 \cos^2 \theta^J, \; Tr [(\omega^K)^2]= -4 \cos^2 \theta^K$  which leads to the expression (\ref{angle of isoclinicity of a 4 dimensional subspace isoclinic w.r.t. an hypercomplex basis 1}).
\end{proof}

 From the similarity invariance of the trace we have the
 \begin{coro} \label{invariance of $(a a' + bb' + cc')$}
 Fixed an admissible basis  $(I,J,K)$,  the quantities $(a a' + bb' + cc'), (a \tilde a + b \tilde b + c \tilde c), (a' \tilde a + b' \tilde b + c' \tilde c)$  are invariant of $U$.
 \end{coro}

\begin{coro}
Let $U \in  \mathcal{IC}^4$.  Fixed an  orthonormal basis $B$, for any $A \in S(\mathcal{Q})$ the matrix of  $\omega^{A}$ is given in  (\ref{matrix of projector omega^A of an isoclinic 4 dimensional subspace})  with the  same choice of sign. In particular if $B$ is a standard basis of some compatible complex structure, the common choice of sign is the upper one.
\end{coro}
\begin{proof}
Fixed an orthonormal basis, the matrices of the forms $\omega^I,\omega^J,\omega^K$ have necessarily the same choice of sign being, from (\ref{angle of isoclinicity of a 4 dimensional subspace isoclinic w.r.t. an hypercomplex basis 2}),   the quantities $(a a' + bb' + cc'), (a \tilde a + b \tilde b + c \tilde c), (a' \tilde a + b' \tilde b + c' \tilde c)$ independent from the row. Consequently, for any $A \in S(\mathcal{Q})$ and a fixed orthonormal basis, the matrix of   $\omega^A= \alpha_1 \omega^I + \alpha_2 \omega^J + \alpha_3 \omega^K$, given in  (\ref{matrix of projector omega^A of an isoclinic 4 dimensional subspace}), has the same choice of sign. The second statement is straightforward.
\end{proof}

 Any unitary vector $X_1 \in U$ is simultaneously a principal vector of the pairs $(U,IU)$, $(U,JU)$, $(U,KU)$.
 Suppose that $U$ is not an orthogonal subspace (w.r.t. $(I,J,K)$)  and let
   \begin{equation} \label{$X_2,Y_2,Z_2$}
  X_2=\frac{I^{-1} Pr^{IU}X_1}{\cos \theta^I}, \; Y_2=\frac{J^{-1} Pr^{JU}X_1}{\cos \theta^J}, \; Z_2=\frac{K^{-1} Pr^{KU}X_1}{\cos \theta^K}
  \end{equation}
 be the unitary vectors such that $(X_1,X_2)$, $(X_1,Y_2)$, $(X_1,Z_2)$ are  (orthonormal) standard bases of the  standard 2-planes $U^I,U^J,U^K$ of $\omega^I,\omega^J,\omega^K$  they generate. The quantities  $<X_1,IX_2>, <X_1,JY_2>,<X_1,KZ_2>$   are the (non negative) cosines of the principal angles of the pairs $(U^I,IU^I),(U^J,JU^J),(U^K,KU^K)$ or equivalently the absolute value of the cosine of the $I,J,K$-K\"{a}hler angles of the 2-planes $U^I,U^J,U^K$ respectively.
 Let $(X_1,X_2,X_3,X_4)$ (resp. $(X_1,Y_2,Y_3,Y_4)$, resp. $(X_1,Z_2,Z_3,Z_4)$) be a standard basis  of the  forms  $\omega^I$ (resp. $\omega^J$, resp. $\omega^K$ ) with the common leading vector $X_1$.

 One has that
 \begin{equation} \label{scalar product $(a a' + bb' + cc')$ in a 4 dimensional isoclinic subspace}
\begin{array}{ccc}
  (a a' + bb' + cc') & = & \cos \theta^I <X_1, JX_2>=\cos \theta^I <X_3, JX_4>= \cos \theta^J <X_1,IY_2>= \cos \theta^J <X_3,IY_4>\\
  (a \tilde a + b \tilde b + c \tilde c) & = &\cos \theta^I <X_1, KX_2>= \cos \theta^I <X_3, KX_4>=\cos \theta^K <X_1, IZ_2>=\cos \theta^K <X_3, IZ_4> \\
  (a' \tilde a + b' \tilde b + c' \tilde c) & = & \cos \theta^J <X_1, KY_2>=\cos \theta^J <X_3, KY_4>= \cos \theta^K  <X_1,JZ_2>=\cos \theta^K  <X_3,JZ_4>.
\end{array}
\end{equation}

 Let moreover  denote by
 \[\xi=<X_2,Y_2>,  \quad \chi=<X_2,Z_2>, \quad \eta=<Y_2,Z_2>\]
   where $\xi, \chi,\eta \in [-1,1]$. In case $U$ is an orthogonal subspace  we
\begin{claim} \label{principal vectors asociated to $X_1$ in case some of the pairs (U,IU), (U,JU),(U,KU) are strictly orthogonal}
Let $X_1 \in U$ unitary. In case $\cos \theta^I=0$ (resp. $\cos \theta^J=0$, resp. $\cos \theta^K=0$), the pair $(Y_2,Z_2)$ (resp. $(X_2,Z_2)$, resp. $(X_2,Y_2)$) is given in   (\ref{$X_2,Y_2,Z_2$}). Any vector of $IU$ (resp. $JU$, resp. $KU$) can be consider a principal vector related to $X_1$. We assume $X_2=Y_2$  (resp. $Y_2=X_2$, resp. $Z_2=X_2$)    then $\xi=1$ (resp. $\xi=1$, resp. $\chi=1$). 

In case  $\cos \theta^J=\cos \theta^K=0$ (resp. $\cos \theta^I=\cos \theta^K=0$, resp. $\cos \theta^I=\cos \theta^J=0$)  then  we assume $X_2=Y_2=Z_2$  with $X_2$  (resp. $Y_2$, resp. $Z_2$)  given in (\ref{$X_2,Y_2,Z_2$}) and then $\xi=\chi=\eta=1$.

Finally in case all three cosines equal zero, $U$ is a r.h.p. subspace. As in previous point we can assume $X_2=Y_2=Z_2$ with $X_2$ any unitary vector orthogonal to $X_1$ and then $\xi=\chi=\eta=1$.
\end{claim}

\begin{prop} \label{invariance of <X_2,Y_2>}
  The cosines  $\xi=<X_2,Y_2>,  \; \chi=<X_2,Z_2>, \; \eta=<Y_2,Z_2>$ are invariants of $U$.
\end{prop}
\begin{proof}
In case $U$ is orthogonal we consider the assumptions of the Claim (\ref{principal vectors asociated to $X_1$ in case some of the pairs (U,IU), (U,JU),(U,KU) are strictly orthogonal}).
One has
\[
\begin{array}{ccc}
  (a a' + bb' + cc') & = &  <X_2,Y_2> \cos \theta^I \cos \theta^J, \\
  (a \tilde a + b \tilde b + c \tilde c) & = &<X_2,Z_2> \cos \theta^I \cos \theta^K, \\
  (a' \tilde a + b' \tilde b + c' \tilde c) & = &<Y_2,Z_2> \cos \theta^J \cos \theta^K.
\end{array}
\]
The first equality is obtained by  substituting in the first of (\ref{scalar product $(a a' + bb' + cc')$ in a 4 dimensional isoclinic subspace})
  $X_2= <X_2,Y_1>Y_1 + <X_2,Y_2>Y_2 +<X_2,Y_3>Y_3 +<X_2,Y_4>Y_4$. The  independence from  $X_1$ follows from  Corollary (\ref{invariance of $(a a' + bb' + cc')$}).
The other equalities follow in the same way.
\end{proof}

Clearly if $U$ is 2-planes decomposable then $\xi,\chi,\eta$  all assume value  $\pm 1$ with $\eta= \xi \cdot \chi$.
It follows that in case of double or triple orthogonality $U$ is clearly a 2-planes decomposable subspace.

Applying the last result, from equation (\ref{angle of isoclinicity of a 4 dimensional subspace isoclinic w.r.t. an hypercomplex basis 2}) we have the following equivalent expression for the angle of isoclinicity
 \begin{equation}\label{equivalent form for the angle of isoclinicity of a 4 dimensional subspace isoclinic w.r.t. an hypercomplex basis}
\begin{array}{l}
\cos^2 \theta^A  = \alpha_1^2 \cos^2 \theta^I + \alpha_2^2\cos^2 \theta^J + \alpha_3^2\cos^2 \theta^K +
  2  \xi \alpha_1 \alpha_2 \cos \theta^I \cos \theta^J + 2 \chi \alpha_1 \alpha_3 \cos \theta^I  \cos \theta^K+ 2 \eta \alpha_2
\alpha_3 \cos \theta^J \cos \theta^K.
\end{array}
\end{equation}

\begin{remk} \label{dependence of $(xi, chi, eta)$ on the admissible basis}
Let $U \in \mathcal{IC}^{2m}, \; m \leq 2$. The cosines   $\xi, \chi, \eta$ are not an intrinsic properties of $U$  i.e. they depend on the  chosen admissible basis.
\end{remk}
\begin{proof}
In case of a 2-plane $U$, let consider the admissible basis $(I,J,K)$  and let $X_2,Y_2,Z_2$ be the usual triple and suppose that $\xi=<X_2,Y_2>=\chi=<X_2,Z_2>=1=\eta=\xi \chi=<Y_2,Z_2>$.

     Let now consider the admissible basis $(-I,-J,K)$ (admissible bases are related by a rotation then they have the same orientation). The new vectors $X_2 \mapsto -X_2$, $Y_2 \mapsto -Y_2$, $Z_2 \mapsto Z_2$,
     then $\xi=1, \chi= -1, \eta=-1$

  An  example for the dimension 4 is given by an $I$-complex 4 dimensional subspace with quaternionic K\"{a}hler angle $\cos \theta$
 (i.e. $\theta$ is the angle of isoclinicity of the pairs $(U,JU=KU)$ (see \cite{articolo_2}). In this case w.r.t. an adapted basis $(I,J,K)$ in (\cite{articolo_2})  we proved that $X_2,Y_2,Z_2$ are mutually orthogonal. Then $\xi=\chi=\eta=0$. If $(I',J',K')$ is another admissible basis with  $I'=\alpha_1 I + \alpha_2 J + \alpha_3 K$, $J'=\beta_1 I + \beta_2 J + \beta_3 K$, $K'=\gamma_1 I + \gamma_2 J + \gamma_3 K$ one has that
 \[\xi'= \alpha_1 \beta_1 (1-\cos \theta), \qquad \chi'= \alpha_1 \gamma_1 (1-\cos \theta), \qquad \eta'= \beta_1 \gamma_1 (1-\cos \theta)\]
 which account for the dependence of  $(\xi',\chi', \eta')$  on  $(\alpha_1, \beta_1,\gamma_1)$.
 \end{proof}

We finally  introduce an intrinsic property of the elements of $\mathcal{IC}^4$. Later  we will see that the statement below is valid for all subspaces $U \in \mathcal{IC}$ regardless their  dimension.
\begin{prop} \label{$S$ sum of square cosined of the angles of isoclinicity (U,IU),(U,JU),(U,KU) of isoclinic subspaces}
Let $(I,J,K)$ be an admissible hypercomplex basis and  $U \in \mathcal{IC}^4$
 with angles  of isoclinicity equal  respectively to $(\theta^I,\theta^J,\theta^K)$. Then $S= \cos^2 \theta^I + \cos^2 \theta^J + \cos^2 \theta^K$  is an intrinsic property of $U$  not depending on the admissible basis.
\end{prop}

\begin{proof}
  It follows directly  from (\ref{equivalent form for the angle of isoclinicity of a 4 dimensional subspace isoclinic w.r.t. an hypercomplex basis}) recalling that any change of admissible basis is represented by a matrix belonging to the special orthogonal group.
  \end{proof}

\subsection{Canonical matrices $C_{IJ},C_{IK}$ of $U \in \mathcal{IC}^4$}  \label{associated chains and matrices} 
It is well known that, if $U$ and $W$ are a pair of $m$-dimensional subspaces of the $n$-dimensional space $V$ and $U^\perp$ and $W^\perp$ the respective orthogonal complements in $V$, the non zero principal angles  of the pair $(U,W)$ are the same as those between $(U^\perp,W^\perp)$. Furthermore, the non zero principal angles between $U$ and $W^\perp$ are the same as those between $U^\perp$ and $W$. We recall some well known properties of the principal angles of a pair of subspaces.

\begin{prop} \cite{ZK} \label{propertiers of CS decomposition}
Let $U$ and $W$ be a pair of respectively $p$ and $q$ dimensional  subspaces  with $p \geq q$ of the $n$-dimensional space $V$ and $U^\perp$ and $W^\perp$ the respective orthogonal complement in $V$.
Denote by $\Theta^\uparrow(U,W)$  (resp. $\Theta^\downarrow(U,W)$) the vector of the $q$ principal angles between  $U$ and $W$ arranged in non decreasing (resp. non increasing) order. One has the following  properties  for the principal angles of the pair $(U,W)$ and of the pair of their orthogonal complements $(U^\perp,W^\perp)$.
\[
\begin{array}{ll}
1) &  \{ \Theta^\downarrow(U,W),0, \ldots,0 \} =\{ \Theta^\downarrow(U^\perp,W^\perp),0, \ldots,0 \},\\
2) &  \{ \Theta^\downarrow(U,W^\perp),0, \ldots,0 \} =\{ \Theta^\downarrow(U^\perp,W),0, \ldots,0 \},\\
3) &  \{ \underbrace{\pi/2, \ldots,\pi/2}_{l}, \{ \Theta^\downarrow(U,W)\} =\{ \pi/2- \Theta^\uparrow(U,W^\perp),0, \ldots,0 \},
\end{array}
\]
where $l=\max{(\dim U - \dim W,0)}$ and extra 0s at the end may need to added on either side to match the sizes.
\end{prop}

In particular if we consider  a 4 dimensional subspace $U$,
from the Proposition (\ref{propertiers of CS decomposition}), it follows that, given a pair of orthogonal decompositions into 2-planes  of  $U$ i.e. $U=U_1  \stackrel{\perp}  \oplus U_2=W_1  \stackrel{\perp}  \oplus W_2$, there exists orthonormal bases $(X_1,X_2)$ of $U_1$, $(X_3,X_4)$ of  $U_2$, $(Y_1,Y_2)$ of $W_1$ and $(Y_3,Y_4)$ of  $W_2$     w.r.t. which   the orthogonal  Gram matrix $G=(<X_i,Y_j>)$  assumes the CS (Cosine-Sine) form:
\begin{equation} \label{CS decomposition}
G=\left(\begin{array} {cccc}
\cos \alpha_1 & 0 & \sin \alpha_1 & 0\\
0 & \cos \alpha_2 & 0 & \sin \alpha_2\\
-\sin\alpha_1 & 0 & \cos \alpha_1 & 0 \\
0 & -\sin \alpha_2 & 0 & \cos \alpha_2
\end{array}
\right)
\end{equation}
with non-negative entries in the upper triangular part.

Such bases are the bases of principal vectors of the pair $(U_1,W_1)$ and $(U_2,W_2)$.
If $\alpha_1 \neq \alpha_2$, the principal vectors are  defined up to sign. The related pairs are then defined up to contemporary change of sign.
Observe that the pairs $(X_1,Y_3)$ and $(X_2,Y_4)$ are pairs of related principal vectors of the subspaces  $(U_1,W_2)$
whereas   $(-X_3,Y_1)$  and $(-X_4,Y_2)$  are pairs of related principal vectors of  $(U_2,W_1)$. In case $\alpha_1 = \alpha_2$, the bases of related principal vectors are  defined up to a common orthogonal transformation in $U_1,U_2,W_1,W_2$.

In the following we determine the \textit{canonical bases} of any $U \in \mathcal{IC}^4$ and   associated \textit{canonical matrices} $C_{IJ}$ and $C_{IK}$  according to the definition given in \cite{Vacpreprint}. We recall that  the  canonical bases are  triples $(\{X_i\},\{Y_i\},\{Z_i\})$ of $\omega^I,\omega^J,\omega^K$-standard bases respectively built up using some determined  procedure. In general  they are  not unique (each triple  may depend for instance on  an arbitrary choice of some  vector to start the procedure); nevertheless they are canonical in the sense that the relative Gram matrices $C_{IJ}=(<X_i,Y_j>),  \;C_{IK}=(<X_i,Z_k>), C_{JK}=(<Y_j,Z_k>)$ are independent from the chosen triple.
According to  the Theorem (\ref{main_theorem 1_ existance of canonical bases with same mutual position}) the knowledge  of $C_{IJ}$ and $C_{IK}$ is necessary to determine the $Sp(n)$-orbit of $U$. In this degenerate case however we cannot apply tout court the procedure described in \cite{Vacpreprint}.
In fact such iterative procedure needs at each step the uniqueness (up to sign) of the first (smallest) principal angle of a pair of standard subspaces of $\omega^I$ and $\omega^J$ and of $\omega^I$ and $\omega^K$. In our case  we have only one standard subspace for any $A \in S(\mathcal{Q})$ which is $U$ itself and, because of the isoclinicity, the  choice of the first pair of  related principal vectors necessary to start the procedure   is not unique.  Here we will show that to overtake this problem  we need  to add one extra condition to such  procedure which will allow us to  determine the canonical bases and associated  pair of canonical matrices $C_{IJ},C_{IK}$.

In the following if $U$ is orthogonal sum of a pair of 2-planes both isoclinic with their $I,J,K$-images with angles $\theta^I,\theta^I,\theta^K$ respectively  we say that $U$ is a \textit{2-planes decomposable} subspace. In a 2-plane decomposable  subspace  the values of $\xi,\chi,\eta$ are all clearly equal to $\pm 1$.

 We consider first the case that none among $\xi,\chi,\eta$ equals $\pm 1$.
    Let $X_1 \in U$ unitary  and   $(X_2,Y_2,Z_2)$ as in  (\ref{$X_2,Y_2,Z_2$}).
The subspaces $L(X_1,X_2),L(X_1,Y_2),L(X_1,Z_2)$ are respectively standard subspaces of $\omega^I, \omega^J,\omega^K$ restricted to $U$.
Consider  $L(X_2,Y_2)$ and the vectors $X_4,Y_4$ of such 2-plane such that $(X_2,X_4$) and $(Y_2,Y_4)$ are   a pair of orthonormal basis consistently oriented   with $(X_2,Y_2)$ (then  $<X_2,Y_4> <0$).

Let then $X_3=-\frac{I^{-1} Pr^{IU} X_4}{\cos \theta^I}$ be the  unique vector such that $<X_3,IX_4>= \cos \theta^I$ i.e. $L(X_3,X_4)$ is an $\omega^I$-standard 2-plane and analogously  $Y_3=-\frac{J^{-1}Pr^{JU}Y_4}{\cos \theta^J}$   the unique vector such that $<Y_3,JY_4>= \cos \theta^J$. Clearly the vectors $X_3$ and $Y_3$ belong to $U$.

Analogously  we consider   $L(X_2,Z_2)$ and the vectors $\tilde X_4,Z_4$ of such 2-plane such that $(X_2,\tilde X_4$) and $(Z_2,Z_4)$ are   a pair of orthonormal basis consistently oriented  with the pair $(X_2,Z_2)$ (then  $<X_2,Z_4> <0$). Again, let $\tilde X_3=-\frac{I^{-1} Pr^{IU} \tilde X_4}{\cos \theta^I}$ be the unique vector such that $<\tilde X_3,I\tilde X_4>= \cos \theta^I$ i.e. $L(\tilde X_3,\tilde X_4)$ is an  $\omega^I$ standard 2-plane and $Z_3=-\frac{K^{-1}Pr^{KU}Z_4}{\cos \theta^K}$  the unique vector such that $<Z_3,K Z_4>= \cos \theta^K$ i.e.  $L (Z_3,Z_4)$ is an  $\omega^K$ standard 2-plane. The vectors $\tilde X_3$ and $Z_3$ belong to $U$. Proceeding in the same way considering the oriented 2-plane $L(Y_2,Z_2)$ we determine the pair $(\tilde Y_4, \tilde Z_4)$ and consequently $(\tilde Y_3,\tilde Z_3)$.
Namely one has
 \[
\begin{array}{llll} \label{oriented bases}
 \text{the pair $(X_4,Y_4) \in L(X_2,Y_2)$ where} \qquad X_4= \frac{Y_2- \xi X_2}{\sqrt{1-\xi^2}}, &  Y_4 =  \frac{-X_2+ \xi Y_2}{\sqrt{1-\xi^2}};\\
\text{the pair $(\tilde X_4,Z_4) \in L(X_2,Z_2)$ where} \qquad \tilde X_4=  \frac{Z_2- \chi X_2}{\sqrt{1-\chi^2}},&  Z_4  =    \frac{-X_2+ \chi Z_2}{\sqrt{1-\chi^2}};\\
\text{the pair $(\tilde Y_4,\tilde Z_4) \in L(Y_2,Z_2)$ where} \qquad  \tilde Y_4=  \frac{Z_2- \mu Y_2}{\sqrt{1-\mu^2}}, & \tilde Z_4  =   \frac{-Y_2+ \mu Z_2}{\sqrt{1-\mu^2}}.
\end{array}
\]

\begin{prop} \label{equality of thirs element $X_3=Y_3$ of the chains} $ $\\
If both $\cos \theta^I \neq 0$ and $\cos \theta^J \neq 0$ one has $X_3= -\frac{I^{-1} Pr^{IU} X_4}{\cos \theta^I}= -\frac{J^{-1}Pr^{JU}Y_4}{\cos \theta^J}=Y_3$,\\
if both $\cos \theta^I \neq 0$ and $\cos \theta^K \neq 0$ one has $\tilde X_3= -\frac{I^{-1} Pr^{IU} \tilde X_4}{\cos \theta^I}= -\frac{K^{-1}Pr^{KU}Z_4}{\cos \theta^K}=Z_3$,\\
if both $\cos \theta^J \neq 0$ and $\cos \theta^K \neq 0$ one has $\tilde Y_3= -\frac{J^{-1} Pr^{IU} \tilde Y_4}{\cos \theta^J}= -\frac{K^{-1}Pr^{KU}\tilde Z_4}{\cos \theta^K}=\tilde Z_3$.
\end{prop}
\begin{proof}
 We prove that $X_3=Y_3$ and  $(X_1,X_2,X_3,X_4)$ and $(X_1,Y_2,X_3,Y_4)$ are a pair of  standard bases of $\omega^I$ and $\omega^J$ restricted to $U=L(X_1,X_2,X_4,X_3)=L(X_1,Y_2,Y_4,Y_3)$ since the other proofs are similar.  From the  Proposition (\ref{invariance of <X_2,Y_2>}), we know that  fixed any  unitary vector $\bar X \in U$, the cosine of the angle between  the unitary vectors $\frac{I^{-1}Pr^{IU} \bar X}{\cos \theta^I}$ and $\frac{J^{-1}Pr^{JU} \bar X}{\cos \theta^J}$ is always the same and equals $\xi$. In particular considering the vector $X_4$, if $\tilde Y_3=-\frac{J^{-1}Pr^{JU}X_4}{\cos \theta^J}$  and  $X_3= -\frac{I^{-1} Pr^{IU} X_4}{\cos \theta^I}$ one has
 $<X_3,\tilde Y_3>= \xi$. Denoting by
 \[
 X_3= -\frac{I^{-1} Pr^{IU} X_4}{\cos \theta^I}, \qquad  \tilde Y_3=-\frac{J^{-1}Pr^{JU}X_4}{\cos \theta^J},\qquad Y_3= -\frac{J^{-1}Pr^{JU}Y_4}{\cos \theta^J}
 \]
 one has
\[
\begin{array} {lll}
\tilde Y_3 & =&- \frac{1}{\cos \theta^J} J^{-1} Pr^{JU}(<X_4,Y_2>Y_2 + <X_4,Y_4>Y_4)\\
&=&-\frac{1}{\cos \theta^J}( <X_4,Y_2> (-\cos \theta^J  X_1) + <X_4,Y_4> (-\cos \theta^J Y_3))=\\
&=& <X_4,Y_2>  X_1 + <X_4,Y_4>  Y_3.
\end{array}
\]
Then
\[
\xi= <X_3,\tilde Y_3>= <X_4,Y_4> <X_3,Y_3>
\]

Since the orthonormal bases $(X_2,X_4)$ and $(Y_2,Y_4)$ of $L(X_2,Y_2)$ are consistently oriented by hypothesis one has $<X_4,Y_4>=<X_2,Y_2>= \xi$
which implies that $<X_3,Y_3>=1$ i.e.
  \[X_3=Y_3.\]
The other equalities of the Proposition are proved in the same way.
\end{proof}


We define  the following  standard bases of $\omega^I|_U$ and $\omega^J|_U$ respectively
\begin{defi} \label{associated chains} Let $U \in \mathcal{IC}^4$, $(I,J,K)$  be an admissible basis and $\theta^I, \theta^J, \theta^K$ the respective angles of isoclinicity. In case none among $\xi,\chi,\eta$ is equal to $\pm 1$  (in particular if $U$ is  neither orthogonal nor 2-planes decomposable), for any unitary $X_1 \in U$, that we call leading vector, we define  the following  standard bases of $\omega^I|_U$ and $\omega^J|_U$ respectively
\[
\begin{array}{l}
\{X_i\}=\{X_1, X_2= \frac{I^{-1} Pr^{IU}X_1}{\cos \theta^I}, X_3=-\frac{I^{-1} Pr^{IU}X_4}{\cos \theta^I}, X_4=  \frac{Y_2- \xi X_2}{\sqrt{1-\xi^2}}\},\\
\{Y_i\}=\{X_1, Y_2= \frac{(J^{-1} Pr^{JU}X_1)}{\cos \theta^J}, Y_3=X_3=-\frac{I^{-1} Pr^{JU}Y_4}{\cos \theta^J}, Y_4  =  \frac{-X_2+ \xi Y_2}{\sqrt{1-\xi^2}} \}
\end{array}
\]
 the  $\omega^I$ and $\omega^J$-\textbf{chains} of $U$  centered on $X_1$,  and the following  standard bases of $\omega^I|_U$ and $\omega^K|_U$ respectively
\[
\begin{array}{l}
\{\tilde X_i\}=\{X_1, X_2=   \frac{I^{-1} Pr^{IU}X_1}{\cos \theta^I}, \tilde X_3=-\frac{I^{-1} Pr^{IU} \tilde X_4}{\cos \theta^I}, \tilde X_4=  \frac{Z_2- \chi X_2}{\sqrt{1-\chi^2}}\},\\
 \{Z_i\}=\{X_1, Z_2=   \frac{K{-1} Pr^{KU}X_1}{\cos \theta^K}, Z_3=\tilde X_3=-\frac{K^{-1} Pr^{KU}  Z_4}{\cos \theta^K}, Z_4  =   \frac{-X_2+ \chi Z_2}{\sqrt{1-\chi^2}}\}
\end{array}
\]
the  $\omega^I$ and  $\omega^K$-\textbf{chains} of $U$ centered on  $X_1$ and the following  standard bases of $\omega^J|_U$ and $\omega^K|_U$ respectively.
\[
\begin{array}{l}
\{\tilde Y_i\}=\{X_1, Y_2=   \frac{J^{-1} Pr^{JU}X_1}{\cos \theta^J}, \tilde Y_3=-\frac{J^{-1} Pr^{JU} \tilde Y_4}{\cos \theta^J}, \tilde Y_4=\frac{Z_2- \eta Y_2}{\sqrt{1-\eta^2}}\},\\
 \{\tilde Z_i\}=\{X_1, Z_2=   \frac{K^{-1} Pr^{KU}X_1}{\cos \theta^K}, \tilde Z_3=\tilde Y_3=-\frac{K^{-1} Pr^{KU} \tilde  Z_4}{\cos \theta^K}, \tilde Z_4  =   \frac{-Y_2+ \eta Z_2}{\sqrt{1-\eta^2}}\}
\end{array}
\]
the  $\omega^J$ and  $\omega^K$-\textbf{chains} of $U$ centered on  $X_1$.
We denote by  $\Sigma (X_1)$  the set of the six chains with leading vector $X_1$.
\end{defi}

Clearly $\Sigma (X_1)$ is  uniquely determined by the leading vector $X_1$. 

 \vskip .5cm

Let consider now the case that at least one among $(\xi,\chi,\eta)$ equals $\pm 1$ (this is always the case if $U$ is orthogonal (claim (\ref{principal vectors asociated to $X_1$ in case some of the pairs (U,IU), (U,JU),(U,KU) are strictly orthogonal}))  or 2-plane-decomposable).
Let denote by $U_1^I=L(X_1,X_2)$, $U_1^J=L(X_1,Y_2)$,  $U_1^K=L(X_1,Z_2)$) the $\omega^I, \omega^J, \omega^K$-standard 2-plane respectively  determined by  $X_1$ and by $(U_1^I)^\perp, (U_1^J)^\perp,(U_1^K)^\perp$ their orthogonal complements in $U$.

Consider first the case that only one among $\xi,\chi, \eta$ equals $\pm 1$. We define $\Sigma(X_1)$ in case $\xi=\pm 1$; the cases $\chi= \pm 1$ and $\eta=\pm 1$ can be treated similarly.
 If $\xi= \pm 1$ one has $U_1^I=L(X_1,X_2)=L(X_1,Y_2)=U_1^J$ and clearly $(U_1^I)^\perp=(U_1^J)^\perp$. Being $X_3 (=Y_3)  \in (U_1^I)^\perp \cap  (U_1^I)^\perp$, in this case any vector $X_3 \in (U_1^I)^\perp$ is a principal vector of the pair $((U_1^I)^\perp, (U_1^J)^\perp)$.
 We assume   $\tilde X_3 (=Z_3) = \tilde Y_3$.  Such unitary vector is determined up to sign.
 Moreover we assume $X_3=\tilde X_3$. It follows that  $Z_4= \tilde Z_4$.   To solve the ambiguity of sign  we  choose $X_4$ in order that the pair $(X_2,Z_2)$) and $(X_2,X_4)$ are consistently oriented.

 Mutatis mutandis, in case $\chi= \pm 1$ or $\eta= \pm 1$   we will have always $X_3(=Y_3)=\tilde X_3 (=Z_3) = \tilde Y_3 (=\tilde Z_3)$ defined up to sign, then  we can  give the following
 \begin{defi} \label{The chains if at least one among the 3 invariants is one or minus one}
 Let $\{X_i \}=(X_1,X_2,X_3,X_4)$. If one among $(\xi,\chi,\eta)$ equals $\pm 1$, in particular if  $U$ has a single orthogonality, we define the following chains:
 \[\xi= \pm 1:\quad \{X_i \}= \{\tilde X_i \}, \quad \{Y_i \}=(X_1,\xi X_2,X_3,\xi X_4)= \{\tilde Y_i \}, \quad \{Z_i \}=(X_1,Z_2,X_3,Z_4)= \{\tilde Z_i \};\]
 \[\chi= \pm 1:\quad \{X_i \}=\{\tilde X_i \}, \quad \{Y_i \}=(X_1,Y_2,X_3,Y_4)= \{\tilde Y_i \},\quad \{Z_i \}=(X_1,\chi X_2,X_3,\chi X_4)= \{\tilde Z_i \};\]
 \[\eta= \pm 1:\quad \{X_i \}= \{\tilde X_i \},\quad \{Y_i \}=(X_1,Y_2,X_3,Y_4)= \{\tilde Y_i \},\quad \{Z_i \}=(X_1,\eta Y_2,X_3,\eta Y_4)= \{\tilde Z_i \}.\]
\end{defi}
Similarly to the choice made when $\xi= \pm 1$, in case $\chi=\pm 1$  or $\eta=\pm 1$,  to solve the ambiguity of sign and have a unique $\Sigma(X_1)$,   we  choose $X_4$ in order that the pair $(X_2,Y_2)$  and $(X_2,X_4)$ are consistently oriented.

Finally,  since  $\eta=\xi  \cdot \chi$, in case a  pair among $(\xi,\chi,\eta)$ are equal to  $\pm 1$  then
all of them are equal to $\pm 1$ (namely either $\eta=\xi= \chi=1 $ or  two of them are equal to  -1 and the other  to 1),  we have the following
 \begin{defi} \label{The chains if all the 3 invariants is one or minus one}
In case  $U$ is a 2-planes decomposable subspace i.e. $\xi,\chi,\eta$ are all equal to  $\pm 1$ we define the following chains:
   \[\{X_i \}= \{\tilde X_i \}, \; \{Y_i \}=(X_1, \xi X_2,X_3,\xi X_4)= \{\tilde Y_i \}, \; \{Z_i \}=(X_1, \chi X_2,X_3,\chi X_4)=(X_1,\eta Y_2,X_3,\eta Y_4)= \{\tilde Z_i \}.\]
   In particular  if  $U$ has a double or triple  orthogonality,  one has   $\{X_i \}=\{\tilde X_i \}=\{Y_i \}= \{\tilde Y_i \}=\{Z_i \}= \{\tilde Z_i \}$.
\end{defi}

\begin{prop} \label{Signa is defined only for non decompasable 2planes}
The function $\Sigma: X_1 \mapsto \Sigma(X_1)$ is defined only if $U$ is not a 2-planes decomposable subspaces.
\end{prop}
\begin{proof}
In case $U$ is not a 2-planes decomposable subspace the vector $X_3(=Y_3)=\tilde X_3 (=Z_3) = \tilde Y_3 (=\tilde Z_3)$ is uniquely defined (after solving the ambiguity in sign as stated beforehand). In case $U$ is a 2-plane decomposable subspace any unitary vector  in $(U_1^I)^\perp$ can be chosen as $X_3$ leading to a different  set of  chains. In fact, for any $X_1 \in U$, although the decomposition of $U$ into standard 2-planes  is independent from the chosen $X_3$,  the  $\omega^I,\omega^J,\omega^K$-standard basis in  $(U_1^I)^\perp= (U_1^J)^\perp=(U_1^K)^\perp$ undergo a rotation.
\end{proof}

In all orthogonal  cases  at least one among  $\xi,\chi,\eta$ equals 1.
 Observe that in case   of double orthogonality  we consider any standard basis  centered on $X_1$   relative to the only non null cosine,
whereas, in case of triple orthogonality of $U$,    $\{X_i\}=\{Y_i\}=\{Z_i\}=\{\tilde X_i\}=\{\tilde Y_i\}=\{\tilde Z_i\}$  is any orthonormal basis of the r.h.p. 2-planes decomposable subspace  $U$ centered on $X_1$.

\begin{prop} \label{orthogonal maps between chains belong to Sp(n)}
Let $U \in \mathcal{IC}^4$   be not   a 2-planes decomposable subspace. The application  mapping  $\Sigma(X) \mapsto \Sigma(Y)$, where $X,Y \in U$ are a pair of leading vectors,  belongs to the real representation of  $Sp(n)$. Moreover $Sp(n)$ is transitive on $\Sigma$.
\end{prop}
\begin{proof}
Let $U \in \mathcal{IC}(4)$   be not   a 2-planes decomposable subspace and consider $\Sigma(X_1)=\{ \{X_i\},\{Y_i\}, \{\tilde X_i\}, \{Z_i\},\{\tilde Y_i\}, \{\tilde Z_i\} \}$.
 Let $C$ be a linear  transformation on $U$. The images $\{CX_i\},\{CY_i\}, \{C \tilde X_i\}, \{CZ_i\}, \{C \tilde Y_i\}, \{C\tilde Z_i\} \}$  are standard bases iff $C$ is an orthogonal map and commutes with $I,J,K$ i.e. iff $C \in Sp(n)$. We prove that they are  chains of $U$ centered on $CX_1= \bar X_1$.
In fact let  $\Sigma(\bar X_1)=\{\{\bar X_i\},\{\bar Y_i\}, \{\tilde {\bar {X_i}}\}, \{\bar Z_i\},\{\tilde {\bar{Y_i}}\}, \{\tilde {\bar{Z_i}}\}    \}$.  The subspaces $L(CX_1,CX_2)$ (resp. $L(CX_1,CY_2)$, $L(CX_1,CZ_2)$) are standard 2-plane and therefore $CX_2=\bar X_2$, (resp. $CY_2=\bar Y_2$, $CZ_2=\bar Z_2$). As an orthogonal transformation, $C$ preserves angles and then $CX_4=\bar X_4$ (resp. $CY_4=\bar Y_4$, $C \tilde Y_4=\tilde {\bar{Y_4}}$, $C \tilde Z_4=\tilde {\bar{Z_4}}$). Consequently $CX_3=\bar X_3$ (resp. $CY_3=CX_3=\bar Y_3$, $CZ_3= \tilde X_3 =\bar Z_3)$ since $L(CX_3,CX_4)=L(\bar X_3, \bar X_4)$  (resp. $L(CY_3,CY_4)=L(\bar Y_3, \bar Y_4)$, $L(C \tilde X_3,C \tilde X_4)=L(\bar{\tilde{X_3}}, \bar{\tilde{X_4}})$, $L(C\tilde X_3,CZ_4)=L(\bar X_3, \bar Z_4)$ being a pair of standard 2-planes with non trivial intersection.
The fact that the action is transitive follows from the   transitivity of $Sp(n)$  on unitary vectors.
\end{proof}

\begin{coro}  \label{Gram matrices of pairs of chains ar invariants of $U$}
Let $X_1 \in U$. For any pair of chains belonging to $\Sigma(X_1)$ the relative Gram matrix is an invariant of $U$.
In particular the angles $\widehat{X_2,Y_2}, \widehat{X_2,Z_2},\widehat{Y_2,Z_2}$ are invariants.
\end{coro}
The last statement follows from the invariance of $<X_4,Y_2>, <\tilde X_4,Z_2>, <\tilde Y_4,Z_2>$.

W.r.t. the chains  $\{X_i\},\{Y_i\}$ and $\{\tilde X_i\},\{Z_i\}$,    the Gram matrices $C_{IJ}=(<X_i,Y_j>)$  and $C'_{IK}=(<\tilde X_i,Z_j>)$ assume the form

\begin{equation} \label{canonical matrix $A^c$ of 4-dimensional isoclinic subspace w.r.t. a non adapted basis}
C_{IJ}=\left(
\begin{array} {cccc}
1 & 0 & 0 & 0\\
0 & \xi & 0 & -\sqrt{1-\xi^2}\\
0 & 0 & 1 & 0 \\
0 &  \sqrt{1-\xi^2} & 0 & \xi
\end{array}
\right); \qquad \qquad
C'_{IK}=\left(
\begin{array} {cccc}
1 & 0 & 0 & 0\\
0 & \chi & 0 & - \sqrt{1-\chi^2}\\
0 & 0 & 1 & 0 \\
0 &  \sqrt{1-\chi^2} & 0 & \chi
\end{array}
\right).
\end{equation}

Observe that such forms differs from the classical  CS decomposition given in (\ref{CS decomposition}). Even if the pairs $(X_2,Y_2)$ and $(X_2,Z_2)$ are related principal vectors (i.e. if $\xi$ and $\chi$ are non negative), the pair $(X_2,Y_4)$ (resp. $(X_2,Z_4)$) consists  of related principal vector only if $\xi=1$ (resp. $\chi=1$).

Clearly $L(X_3,X_4)=L(\tilde X_3,\tilde X_4)$. The bases $(X_3,X_4)$ and $(\tilde X_3,\tilde X_4)$, being $\omega^I$-standard bases, are consistently oriented. Let
 \[
 C: \left(   \begin{array}{cc}
 <X_3,\tilde X_3> & <X_3,\tilde X_4>\\
<X_4,\tilde X_3> & <X_4,\tilde X_4>
\end{array}      \right)
 = \left(   \begin{array}{cc}
 \Gamma & -\Delta\\
\Delta & \Gamma
\end{array}
\right)
\]
the orthogonal matrix of the change of basis. 
The orthogonal matrices $C_{IJ}=(<X_i,Y_j>)$ and  $C_{IK}=(<X_i,Z_j>)$ of the relative position of the basis $\{X_i\}=(X_1,X_2,X_3,X_4)$,  $\{Y_i\}=(X_1,Y_2,X_3,Y_4)$ and $\{Z_i\}=(X_1,Z_2,\tilde X_3,Z_4)$   are given by
 \begin{equation} \label{canonical matrices $C_{IJ}$ $C_{IK}$ and of 4-dimensional isoclinic subspace w.r.t. the associated chains}
C_{IJ}=\left(
\begin{array} {cccc}
1 & 0 & 0 & 0\\
0 & \xi & 0 & - \sqrt{1-\xi^2} \\
0 & 0 & 1 & 0 \\
0 &  \sqrt{1-\xi^2} & 0 & \xi
\end{array}
\right), \qquad
  C_{IK}=\left(
  \begin{array} {cccc}
1 & 0 & 0 & 0\\
0 & \chi & 0 & - \sqrt{1-\chi^2}\\
0 & -\Delta \sqrt{1-\chi^2} &\Gamma & -\chi \Delta \\
0 &  \Gamma  \sqrt{1-\chi^2} & \Delta & \chi \Gamma.
\end{array}
\right).
\end{equation}

To determine $\Gamma=<X_3,\tilde X_3>$,  being $<Y_2,Z_2>= <Y_2,X_2><Z_2,X_2>+  <Y_2,X_4><Z_2,X_4>$,
 we get $\eta= \xi \chi  + \sqrt{1-\xi^2}\sqrt{1-\chi^2} \; \Gamma$. From the above expression, in case neither $\xi$ nor $\chi$ equal 1, we get:
  \begin{equation} \label{cos phi}
      \Gamma= \frac{\eta-\xi \chi}{\sqrt{1-\xi^2} \; \sqrt{1-\chi^2} } \in [-1,1].
    \end{equation}

\begin{prop} \label{Gamma is an invariant of an isoclinic 4 dimensional subspace 1}
If none among $\xi,\chi,\eta$ is equal to $\pm 1$
 the value of   $\Gamma \in [-1,1]$ is  given in (\ref{cos phi}). 
 If instead at least one among $\xi,\chi,\eta$ is equal to $\pm 1$ then $\Gamma=1$.  In particular this happens if  $U$ is orthogonal or is a 2-planes decomposable subspace.
  In all cases,  the pair $(\Gamma,\Delta)$ is an invariant of $U$.   
\end{prop}

\begin{proof}
In case none among $\xi,\chi,\eta$ is equal to $\pm 1$, the invariance of the pair $(\Gamma,\Delta)$ follows from Corollary  ( \ref{Gram matrices of pairs of chains ar invariants of $U$}).
In case at least one among $\xi,\chi,\eta$ is equal to $\pm 1$, by construction in all chains given in the Definitions
 (\ref{The chains if at least one among the 3 invariants is one or minus one}) and (\ref{The chains if all the 3 invariants is one or minus one})
 it is  $X_3=\tilde X_3$. We remark that if   $\eta=-1$ and none among $\xi,\chi$ equals $\pm 1$,  the (\ref{cos phi}) is perfectly defined giving   $\Gamma=-1$. The difference in sign follows from the different sign of $\tilde X_4$ in the two constructions.
 \end{proof}

From Corollary(\ref{Gram matrices of pairs of chains ar invariants of $U$}), which states in particular   the   invariance of the triple $(\xi,\chi,\eta)$ and of the pair $(\Gamma, \Delta)$,  one has the
\begin{prop} \label{invariance of $C_{IJ}$ and $C_{IK}$ w.r.t. any leading vector}
The matrices $C_{IJ}$ and $C_{IK}$  given in  (\ref{canonical matrices $C_{IJ}$ $C_{IK}$ and of 4-dimensional isoclinic subspace w.r.t. the associated chains})
w.r.t. the   chains $\{X_i \}, \{Y_i \}$ and $\{X_i \}, \{Z_i \}$ centered on a common leading vector  are invariant of $U$.
\end{prop}

 We underline the following fact.
\begin{remk}
The pair  $(\Gamma,\Delta)$  in not an intrinsic property of $U$.   
  \end{remk}
 As an  example to show the dependance of $(\Gamma,\Delta)$  from the admissible basis, we consider the $I$-complex 4-dimensional subspace with quaternionic K\"{a}hler angle $\cos \theta$ of Remark (\ref{dependence of $(xi, chi, eta)$ on the admissible basis}).  In that case, if $ \alpha_1=0$  it is  $\xi=\chi=0$ and $\Gamma= \eta=\beta_1 \gamma_1 (1- \cos \theta)$ which clearly depends  on $\beta_1$ and $\gamma_1$.

Since if   $U$ is orthogonal one has $(\Gamma,\Delta)=(1,0)$,  we can state the following
\begin{coro} \label{some of the pair $(U,IU), (U,JU),(U,KU)$ is strictly orthogonal}
In case one among $\xi,\chi,\eta$ is equal to $\pm 1$,  the matrices $C_{IJ}$ and $C_{IK}$ are given in (\ref{canonical matrix $A^c$ of 4-dimensional isoclinic subspace w.r.t. a non adapted basis}) i.e.  $C_{IK}= C'_{IK}$.
In particular
\begin{enumerate}
\item If  $U$ is $I$ or $J$-orthogonal   one has   $C_{IJ}= Id$.
\item If  $U$ is $K$-orthogonal one has  $C_{IK}= Id$.
\item In case of double or triple orthogonality  one has  $C_{IJ}=C_{IK}= Id$.

\end{enumerate}
\end{coro}

Following the definition  given in \cite{Vacpreprint}, we give the
\begin{defi} \label{definiton of canonical bases and matrices of Is(4)}
 Let $U \in \mathcal{IC}^4$. Fixed  an admissible basis $(I,J,K)$, for any leading vector $X_1$, we call the chains $\{X_i\},\{Y_i\},\{Z_i\}$ (resp. the matrices $C_{IJ}$ and $C_{IK}$)  determined above \textbf{canonical bases}  (resp. \textbf{canonical matrices}) of the subspaces  $U \in \mathcal{IC}^4$.
\end{defi}
 Clearly for any leading vector we have  a different set of canonical bases. As explained  beforehand,   we denote them "canonical" since, by the invariance of  $(\xi,\chi,\eta,\Delta)$,  the matrices $C_{IJ}$ and $C_{IK}$ are invariants of $U \in \mathcal{IC}^4$ having to the unique forms  given in (\ref{canonical matrices $C_{IJ}$ $C_{IK}$ and of 4-dimensional isoclinic subspace w.r.t. the associated chains})
 regardless the leading vector $X_1$.  We summarize the results obtained in the following
\begin{prop} \label{canonical matrices of any  subspace of ${IC}^(4)$}
Fixed an admissible basis $(I,J,K)$, to any $U \in \mathcal{IC}^4$   we can associate  the orthogonal canonical matrices $C_{IJ}$ and $C_{IK}$   given in
(\ref{canonical matrices $C_{IJ}$ $C_{IK}$ and of 4-dimensional isoclinic subspace w.r.t. the associated chains})
 representing the mutual position of the canonical (standard) bases $\{X_i\},\{Y_i\},\{Z_i\}$ of $\omega^I,\omega^J,\omega^K$. Such matrices depend  on
the triple of invariants $(\xi,\chi,\eta)$ and on the sign of $\Delta=<X_4,\tilde X_3>= \pm \sqrt{1-\Gamma^2}$  where $\Gamma= \Gamma(\xi,\chi,\eta)$ is  given in (\ref{cos phi}) if  none among $\xi,\chi,\eta$ is equal to $\pm 1$  else $\Gamma=1$. The second case happens in particular    if  $U$ is orthogonal or a 2-planes decomposable subspace. 
\end{prop}
Then according to Theorem (\ref{main_theorem 1_ existance of canonical bases with same mutual position}) we state the
\begin{teor}
The invariants $(\xi,\chi,\eta, \Delta)$ together with the angles $(\theta^I, \theta^J, \theta^K)$ determine the orbit of any $U \in \mathcal{IC}^4$. In particular if $U$ is orthogonal or 2-planes decomposable (in which case $(\Gamma,\Delta)=(1,0)$) the first set reduces to  the pair $(\xi, \chi)$.
\end{teor}

In particular, from the  Definition (\ref{The chains if at least one among the 3 invariants is one or minus one}) 
and the Corollary (\ref{some of the pair $(U,IU), (U,JU),(U,KU)$ is strictly orthogonal}) one has
\begin{prop} \label{orbit of orthogonal 4-dimensional subspaces} $ $
\begin{enumerate}
\item In case $U$ is $I$-orthogonal (resp. $J$-orthogonal, resp. $K$-orthogonal),
     the $Sp(n)$-orbit is characterized by  the triple  $(\chi, \theta^J,  \theta^K)$ (resp. $(\chi, \theta^I, \theta^K)$, resp. $(\xi, \theta^I, \theta^J)$).
\item \label{double orthogonality} In case of double orthogonality of $U$,
    the $Sp(n)$-orbit depends on  the only  non null cosine . 
\item \label{triple orthogonality} All subspaces with a triple orthogonality (i.e. all r.h.p. subspaces) share the same $Sp(n)$-orbit.
\end{enumerate}
\end{prop}

We complete the analysis of 4-dimensional isoclinic subspaces,  by giving some examples of orthogonal and non orthogonal cases.

Let consider an  \textbf{$I$-complex 4-dimensional subspace $U$}.
In \cite{articolo_2} we proved that $AU$ is the same   for any $A \in I^\perp \cap S(\mathcal{Q})$ and we called \textit{$I^\perp$-K\"{a}hler angle}  and denoted by  $\theta^{I^\perp}$ the angle of isoclinicity of the pair $(U,AU)$. Then w.r.t. an adapted basis $(I,J,K)$ one has $(\cos \theta^I,\cos \theta^J,\cos \theta^K)=(1,\theta^{I^\perp}, \theta^{I^\perp})$.  This is a typical elements of the set $\mathcal{IC}^4$.  In particular in the first example we consider the non orthogonal  case  i.e. the case when the $I^\perp$-K\"{a}hler angle of $U$  is not $\pi/2$. The results obtained thereof apply also to a \textbf{4-dimensional quaternionic subspace} which is a particular case of a non totally  $I$-complex subspace  characterized by  $(\cos \theta^I= \cos \theta^J= \cos \theta^K=1)$.
W.r.t. any admissible basis in the quaternionic case  $Pr^{AU} X_1=X_1, \; \forall A \in S(\mathcal{Q}), \; \forall X \in U$ then  $X_2=-IX_1, \; Y_2=-JX_1, \; Z_2=-KX_1$ which implies  $\xi=\chi=\eta=0$.

More generally, in \cite{articolo_2} we proved that the triple $(X_2,Y_2,Z_2)$ of a 4-dimensional  $I$-complex subspace $U$ is orthonormal iff the admissible basis is an adapted one.
  Then, w.r.t. any adapted basis $(I,J,K)$, in both $I$-complex and quaternionic case, $\xi=\chi=\eta=0$ which implies $\Gamma=<X_3,\tilde X_3>=0$.
To compute $\Delta$, we apply the formulas given in section (\ref{oriented bases}). One has
\[X_4=Y_2, \; Y_4= -X_2, \; \tilde X_4= Z_2,  \; Z_4= -X_2.\]
Furthermore \[\tilde X_3= -I^{-1} Pr^{IU}\tilde X_4=I\tilde X_4= IZ_2.\]
\[Y_2=\frac{J^{-1}Pr^{JU}X_1}{\cos \theta}=\frac{J^{-1}Pr^{KU}X_1}{\cos \theta}=J^{-1}KZ_2= -IZ_2.\] Then
\[ \Delta= <X_4,\tilde X_3>= -<IZ_2,IZ_2>=-1.\]

Therefore the set $(\xi,\chi,\eta,\Gamma,\Delta)=(0,0,0,0,-1)$ is an  invariant (resp. an intrinsic property) of an $I$-complex subspace (resp. quaternionic subspace).
The  chains $\{X_i \},\{Y_i \},\{\tilde X_i \},\{Z_i \}$ of an $I$-complex subspace with leading vector $X_1 \in U$ w.r.t. the adapted basis  $(I,J,K)$ are:
\[
\begin{array}{lllllll}
\{X_i \} &=&\{X_1,X_2,X_3,X_4 \}&=& \{X_1,-IX_1,Z_2,-IZ_2 \}&=&\{X_1,X_2,Z_2,Y_2\} \\
\{Y_i \}&=&\{X_1,Y_2,X_3,Y_4 \}&=&\{X_1,-IZ_2,Z_2,IX_1 \}&=&\{X_1,Y_2,Z_2,-X_2\}\\
\{\tilde X_i \}&=&\{X_1,X_2,\tilde X_3,\tilde X_4 \}&=&\{X_1,-IX_1,IZ_2,Z_2 \}&=&\{X_1,X_2,-Y_2,Z_2 \}\\
\{Z_i \}&=&\{X_1,Z_2,X_3,Z_4 \}&=&\{X_1,Z_2,IZ_2,IX_1\}&=&\{X_1,Z_2,-Y_2,-X_2\}
\end{array}
\]

In  particular for a quaternionic subspace it is $Y_2= -JX_1$ and $Z_2=-KX_1$ then
\[
\begin{array}{lll}
\{X_i \} &=& \{X_1,-IX_1,-KX_1,-JX_1 \}\\
\{Y_i \}&=&\{X_1,-JX_1,-KX_1,IX_1 \}\\
\{\tilde X_i \}&=&\{X_1,-IX_1,JX_1,-KX_1 \} \\
\{Z_i \}&=&\{X_1,-KX_1,JX_1,IX_1 \}
\end{array}
\]

W.r.t. the canonical bases $\{X_i \},\{Y_i \},\{Z_i \}$  the canonical matrices
(\ref{canonical matrices $C_{IJ}$ $C_{IK}$ and of 4-dimensional isoclinic subspace w.r.t. the associated chains})
 of an $I$-complex 4-dimensional subspace are
 \[
C_{IJ}=C'_{IK}
=\left(
\begin{array} {cccc}
1 & 0 & 0 & 0\\
0 & 0 & 0 & -1 \\
0 & 0 & 1 & 0 \\
0 & 1 & 0 & 0
\end{array}
\right),
\qquad \qquad
  C_{IK}=\left(
\begin{array} {cccc}
1 & 0 & 0 & 0\\
0 & 0 & 0 & -1\\
0 & 1   & 0 & 0 \\
0 & 0 & -1 & 0
\end{array}
\right).
\]

We give the following  Proposition whose result is in accordance with what stated  \cite{articolo_2}.
\begin{prop} All 4-dimensional quaternionic subspaces are in the same $Sp(n)$-orbit, whereas the $I^\perp$-K\"{a}hler angle $\theta^{I^\perp}=\theta^J=\theta^K$
determines the $Sp(n)$-orbit of an $I$-complex subspace where $(I,J,K)$ is an adapted basis.
\end{prop}

We give a pair of examples for the  orthogonal case where we recall $(\Gamma,\Delta)=(1,0)$:  we  consider \textbf{4-dimensional totally complex subspaces} and \textbf{4-dimensional real Hermitian product  subspaces}.

In the first case $(\cos \theta^I=1, \cos \theta^J=\cos \theta^K=\theta^{I^\perp}=0)$ then the 4-dimensional totally complex subspaces represent a case of  double orthogonality. According to point (\ref{double orthogonality}) of the Proposition (\ref{orbit of orthogonal 4-dimensional subspaces}), $\xi=\chi=\eta=1$, $C_{IJ}= C_{IK}= Id$ and the $Sp(n)$-orbit is determined by the only non null cosine of the angle of isoclinicity which equals 1 for all totally complex subspaces.   Therefore
\begin{prop}
  Let $I \in S(\mathcal{Q})$.  All 4-dimensional  totally $I$-complex subspaces share the same $Sp(n)$-orbit.
\end{prop}
 This is accordance with the result that appear in \cite{articolo_2}. We give an additional prove of this fact.
In \cite{Vacpreprint} we proved that a pair of subspaces $U,W$  are in the same orbit if there exist a basis in $U$ and one in $W$ w.r.t. which the  matrices $H$ of the Hermitian products  are the same.  W.r.t. a standard basis $B=\{X_1,X_2,X_3,X_4\}$ of $\omega^I|_U$, computing the Hermitian product given in  (\ref{Hermitian product in an Hermitian quaternionic vector space}), one has
 \[
H|_U=\left(
 \begin{array} {cccc}
 1 & i & 0 & 0\\
 -i & 1 & 0 & 0\\
 0 & 0 & 1 & i\\
 0 & 0 & -i & 1
  \end{array}
  \right)
 \]
 which is clearly  independent from $B$.

The case of a \textbf{4-dimensional real Hermitian product  subspaces} (r.h.p.s.), is the only  case of triple orthogonality being $(\cos \theta^I,\cos \theta^J,\cos \theta^K)=(0,0,0)$.
Applying the results stated in point (\ref{triple orthogonality})  of the Proposition (\ref{orbit of orthogonal 4-dimensional subspaces}), the canonical matrices $C_{IJ}$ and $C_{IK}$ of all r.h.p.  subspaces w.r.t. any admissible basis $(I,J,K)$ and any orthonormal (standard and canonical)  basis are  the identity matrix.  We have then  the confirmation that
\begin{prop} \label{canonical bases of r.h.p.subspaces}
 All 4-dimensional r.h.p. subspaces share the same $Sp(n)$-orbit.
\end{prop}


\section{Isoclinic subspaces of dimension $2m>4$}

The purpose of this article is to determine the $Sp(n)$-orbit  of an isoclinic subspace $U \in \mathcal{IC}^{2m}$  in the real Grassmannian $G^\R(2m,4n)$.  After having studied the 2 and 4-dimensional cases, in the following we consider $m>2$. We prove that, similar to the result achieved in dimension 4,  to any $U \in \mathcal{IC}^{2m}, \; m>2$ we can associate  two sets of invariants: a triple   $(\xi,\chi,\eta)$ and the pair $(\Gamma,\Delta)$ which determine the canonical form of the matrices $C_{IJ},C_{IK}$. In fact, since $\Gamma$ is a function of $(\xi,\chi,\eta)$,  according to Theorem (\ref{main_theorem 1_ existance of canonical bases with same mutual position}) appeared in  \cite{Vacpreprint},  the vector  $\{\xi,\chi,\eta,\Delta\}$,  and the triple $(\theta^I,\theta^J,\theta^K)$ of the angles of isoclinicity w.r.t. a fixed admissible basis $(I,J,K)$ determine the $Sp(n)$-orbit of $U$.

\subsection{The triple of invariants  $(\xi,\chi,\eta)$} 
Let $U \in \mathcal{IC}^{2m}$  and  $(\theta^I,\theta^J,\theta^K)$ be the  angles of isoclinicity w.r.t. the admissible basis $(I,J,K)$.
   Let  $X_1 \in U$   be a generic unitary vector.  If $U$ is not orthogonal the triple ($X_2,Y_2,Z_2)$ is given in (\ref{$X_2,Y_2,Z_2$})
      otherwise it is obtainable  from  Claim (\ref{principal vectors asociated to $X_1$ in case some of the pairs (U,IU), (U,JU),(U,KU) are strictly orthogonal}).
Generalizing the result of  the Proposition (\ref{invariance of <X_2,Y_2>}), one has the following

\begin{prop} \label{invariance of xi,chi,eta of U in {IC}}
Let $U \in \mathcal{IC}^{2m}$. For any vector $X_1 \in U$ and fixed an admissible basis $(I,J,K)$ the triple $(\xi=<X_2,Y_2>,\chi=<X_2,Z_2>,\eta=<Y_2,Z_2>)$   is an invariant of $U$. Such triple is not  an  intrinsic property of $U$.
\end{prop}

\begin{proof}

In case $U$ is orthogonal we first consider the assumptions of the Claim (\ref{principal vectors asociated to $X_1$ in case some of the pairs (U,IU), (U,JU),(U,KU) are strictly orthogonal}).
Let $A= \alpha_1 I + \alpha_2 J +\alpha_3 K,  \; A   \in S(\mathcal{Q})$. Calculating as in the proof of the position  (\ref{isoclinicity w.r.t. one hypercomplex basis implies isoclinicity w.r.t. any compatible complex structure}) regarding the 4-dimensional case, and considering for example the first row of the matrices of $\omega^I,\omega^J,\omega^K$ w.r.t. some orthonormal basis  of $U$ whose 
entries we denote by $a_i,a'_i,\tilde a_i, \; i= 1, \ldots, 2m $ respectively, one has
\begin{equation} \label{angle of isoclinicity of a 2m  dimensional isoclinic subspace}
\begin{array}{lll}
\cos^2 \theta^A & = &\alpha_1^2 \cos^2 \theta^I + \alpha_2^2\cos^2 \theta^J + \alpha_3^2\cos^2 \theta^K + 2 \alpha_1 \alpha_2( \sum_{i=1}^{2m}a_i a_i')+
2 \alpha_1 \alpha_3(\sum_{i=1}^{2m}a_i \tilde a_i)+ 2 \alpha_2
\alpha_3( \sum_{i=1}^{2m}a'_i \tilde a_i).
\end{array}
\end{equation}
By the hypothesis of isoclinicity,  for any $A \in \mathcal{Q}$, the quantities
$\sum_{i=1}^{2m}a_i a_i', \; \sum_{i=1}^{2m}a_i \tilde a_i, \;\sum_{i=1}^{2m}a'_i \tilde a_i$ do not depend on $X_1$. 
As in the 4-dimensional case in fact we can assume that one among $\alpha_1,\alpha_2,\alpha_3$ is zero and the result follows.

 Considering in particular a triple of standard bases $\{X_i\}, \{Y_i\}, \{Z_i\}$ centered on a common vector  $X_1=Y_1=Z_1$  it is straightforward to verify that
\[
\begin{array}{ccc}
  \sum_{i=1}^{2m}a_i  a'_i & = & \cos \theta^I <X_1, JX_2>= \cos \theta^J <X_1,IY_2>= <X_2,Y_2> \cos \theta^I \cos \theta^J, \\
  \sum_{i=1}^{2m}a_i \tilde a_i & = &\cos \theta^I <X_1, KX_2>= \cos \theta^K <X_1, IZ_2>=<X_2,Z_2> \cos \theta^I \cos \theta^K, \\
  \sum_{i=1}^{2m}a'_i \tilde a_i & = & \cos \theta^J <X_1, KY_2>= \cos \theta^K  <X_1,JZ_2>=<Y_2,Z_2> \cos \theta^J \cos \theta^K.
\end{array}
\]

For instance  for the first equality
\[
\cos \theta^I<X_1,JX_2> = \cos \theta^I<Y_1,J(\sum_{i=1}^{2m} <X_2,Y_i>Y_i)>= \cos \theta^I \cos \theta^J <X_2,Y_2>
\]
and the stated invariance of the triple $(\xi,\chi,\eta)$ from the leading vector is proved.

As an example of the dependence of $(\xi, \chi, \eta)$ on the admissible basis  we can consider  a $2m$-dimensional  $I$-complex subspace. In \cite{articolo_2} we showed that it admits a decomposition into 4-dimensional $I$-complex  subspaces. Such addends are mutually Hermitian orthogonal i.e. they belong to orthogonal quaternionic  subspaces.  Then the conclusion follows from Remark (\ref{dependence of $(xi, chi, eta)$ on the admissible basis}).
\end{proof}
We will show afterwards in the Proposition (\ref{independence of canonical matrices from the leading vector}) that, like in the 4-dimensional case,   the content of the Proposition (\ref{invariance of xi,chi,eta of U in {IC}}) is part of a more general result.

\subsection{The  associated  subspaces of type $U^{IJ}, U^{IK}, U^{JK}$ of $U \in \mathcal{IC}^{2m}$} 

Let $U \in \mathcal{IC}^{2m}$ and  $(\theta^I,\theta^J, \theta^K)$ be the angles of isoclinicity w.r.t. the admissible basis $(I,J,K)$.
\begin{defi} \label{subspaces of type $U^{IJ}, U^{IK}, U^{IJ}$}
We say that a subspace $W \subset U \in \mathcal{IC}^{2m}$ is of type $U^{IJ}$ (resp. of type $U^{IK}$, resp.  of type $U^{JK}$) if the pair $(W,IW)$ and $(W,JW)$ (resp. $(W,IW)$ and $(W,KW)$, resp. $(W,JW)$ and $(W,KW)$) are isoclinic with angles $(\theta^I,\theta^J)$ (resp. $(\theta^I,\theta^K)$, resp. $(\theta^J,\theta^K)$).
\end{defi}

When dealing with the 4-dimensional case,  to any leading vector $X_1 \in U$ we associated a set of six chains.
According to the Proposition (\ref{Signa is defined only for non decompasable 2planes}), only if $U$ is not a 2-planes decomposable subspace we defined the vector $\Sigma(X_1)$  whose entries are exactly the six chains  given   in the  Definition (\ref{associated chains}) if none among $\xi,\chi,\eta$ is equal to $\pm 1$
  or in Definition   (\ref{The chains if at least one among the 3 invariants is one or minus one}) if only one among $\xi,\chi,\eta$ is equal to $\pm 1$ .
 In these cases in fact the  correspondence $X_1 \mapsto \Sigma(X_1)$ is one to one. Moreover we recall that in the second case, which happens in particular is $U$ has a single orthogonality, one has $\Gamma=1$ (see the Proposition (\ref{Gamma is an invariant of an isoclinic 4 dimensional subspace 1})).

Applying the same procedure described in  subsection (\ref{associated chains and matrices}), fixed a unitary vector $X_1 \in U$, we  build the   chains  given   in  the Definitions (\ref{associated chains}) and 
(\ref{The chains if at least one among the 3 invariants is one or minus one}) in case $U$ is not 2-planes decomposable and the ones defined  in  the Definition (\ref{The chains if all the 3 invariants is one or minus one})  otherwise.
The result of the Proposition (\ref{equality of thirs element $X_3=Y_3$ of the chains}) is clearly still valid 
i.e.
\[X_3=Y_3, \quad \tilde X_3= Z_3, \quad \tilde Y_3=\tilde Z_3\]
In particular if at least one among $\xi,\chi,\eta$ equals $\pm 1$, one has that  $X_3=Y_3=\tilde X_3= Z_3= \tilde Y_3=\tilde Z_3$ which implies that $(\Gamma,\Delta)=(1,0)$.

Denoting  by $\{X_i\}_\R, \{Y_i\}_\R,\{\tilde X_i\}_\R,\{Z_i\}_\R,\{\tilde Z_i\}_\R,\{\tilde Z_i\}_\R$ the linear span of the chains built on a common leading vector, one has that,
if $\dim U=4$,  $\{X_i\}_\R= \{Y_i\}_\R=\{\tilde X_i\}_\R=\{Z_i\}_\R=\{\tilde Z_i\}_\R=\{\tilde Z_i\}_\R=U$.

 This is not in general the case when  $2m > 4$. By construction however $\{X_i\}_\R= \{Y_i\}_\R$ and then,  according to the Definition (\ref{subspaces of type $U^{IJ}, U^{IK}, U^{IJ}$}), such subspace is of type $U^{IJ}$. We will  denote it by $U^{IJ}(X_1)$.
   Analogously by construction  $\{\tilde X_i\}_\R= \{\Z_i\}_\R$ is of type $U^{IK}$ and  $\{\tilde Y_i\}_\R= \{\tilde Z_i\}_\R$ is of type $U^{JK}$. We will denote them by  $U^{IK}(X_1)$
  and  $U^{JK}(X_1)$ respectively.

\begin{prop}  \label{orthogonal subspaces belong to mathcal{IC} 4 with angles theta I,theta J,theta K}
 If at least one among $(\xi,\chi,\eta)$  is equal to $\pm 1$, in particular if $U$ is orthogonal,   one has $U^{IJ}(X_1)=U^{IK}(X_1)=U^{JK}(X_1) \in \mathcal {IC}^4$ with same angles  of isoclinicity of $U$.
 Furthermore if  only one among $(\xi,\chi,\eta)$  is equal to $\pm 1$, in particular if $U$ has a single orthogonality, such subspace is unique. This is not the case if $U$ is 2-planes decomposable.
 \end{prop}

\begin{proof}
If at least one among $(\xi,\chi,\eta)$  is equal to $\pm 1$, from the Definition (\ref{The chains if at least one among the 3 invariants is one or minus one}) and (\ref{The chains if all the 3 invariants is one or minus one}),  one has  $\{X_i\}= \{\tilde X_i\}$  and the first statement follows. In case  $U$  is not 2-planes decomposable (in particular if $U$ has a single orthogonality)   the unicity follows from the unicity of $\Sigma(X_1)$.
 Otherwise (in particular if $U$ has a double or triple orthogonality) let $X_3, \bar X_3 \in (U_1^I)^\perp \cap U$ and denote by $U_2^I= L(X_3,X_4), \; \bar U_2^I= L(\bar X_3,\bar X_4)$ the $\omega^I, \omega^J,\omega^K$-standard 2-planes containing $X_3$ and $\bar X_3$ respectively.  One has that $U_1^I \oplus U_2^I$ and $U_1^I \oplus \bar U_2^I$ are a pair  of  different 4 dimensional isoclinic subspace unless $\bar X_3 \in L(X_3,X_4)$.
 \end{proof}

\begin{defi}
Let $U \in \mathcal{IC}^{2m}$. 
We call \textbf{associated subspaces} of $U$ of type $U^{IJ}$ and $U^{IK}$  (resp. of type $U^{IJ}$ and $U^{JK}$, resp. $U^{JK}$ and $U^{IK}$) a pair  of 4-dimensional subspaces of type $U^{IJ}$ and $U^{IK}$ (resp. of type $U^{IJ}$ and $U^{JK}$, resp. of type $U^{JK}$ and $U^{IK}$) generated by the $\omega^I$ and $\omega^J$ (resp. $\omega^I$ and $\omega^K$, resp. $\omega^J$ and $\omega^K$) chains centered on a common vector.
\end{defi}

Clearly if one among $\xi,\chi,\eta$ equals $\pm 1$  and $X_1  \in U$, any pair among $U^{IJ}(X_1), U^{IK}(X_1),U^{JK}(X_1)$ are associated subspaces.
In the following we will need only associated subspaces of type $U^{IJ}$ and $U^{IK}$. Anyway all properties given in the next proposition  for subspaces of one kind (for example $U^{IJ}$) or for a pair of subspaces of different kind (for instance   $U^{IJ }$ and $U^{IK}$) can be extended to all 3 kinds or to all three pairs.


\begin{prop} \label{properties of subspaces of tupe $U^{IJ}$, $U^{IK}$, $U^{JK}$}
 Let $U \in \mathcal{IC}^{2m}$. 
\begin{enumerate}
\item  \label{invariance of $U^{IJ}(X_1)$ from the leading vector X di $U^{IJ}(X_1)$}
 Let $U \in \mathcal{IC}^{2m}$ be not 2-planes decomposable.
Let $X \in U$ and consider   $U^{IJ}(X)$.  
 Then for any $X_1 \in U^{IJ}(X)$ 
 one has $U^{IJ}(X)=U^{IJ}(X_1)$. 

\begin{proof}
 If none among $\xi,\chi,\eta$ equals $\pm 1$, 
 for any $X_1 \in U^{IJ}(X)$
 the standard 2-planes $L(X_1,X_2=\frac{I^{-1}P^{IU}X_1}{\cos \theta^I})$ and $L(X_1,Y_2=\frac{J^{-1}P^{JU}X_1}{\cos \theta^J})$
  are both in $U^{IJ}(X)$.  Then also the pair $(X_4,Y_4)$   belonging to $L(X_2,Y_2)$   is in $U^{IJ}(X)$    and consequently $X_3=Y_3$    is in $U^{IJ}(X)$.
    Mutatis mutandis, the same proof can be used for $U^{IK}(X)$ or $U^{JK}(X)$.
    In case only one   among $\xi,\chi,\eta$ equals $\pm 1$, 
    suppose without lack of generality that 
    $\chi= \pm 1$ and consider the uniquely determined $U^{IJ}(X_1)$. Applying the proof above, one has that  $U^{IJ}(X_1)=U^{IJ}(X)$. By the Proposition (\ref{orthogonal subspaces belong to mathcal{IC} 4 with angles theta I,theta J,theta K}) one has $U^{IJ}(X_1)=U^{IK}(X_1)=U^{JK}(X_1)$ and we can extend such property to $U^{IK}(X)$ and $U^{JK}(X)$.
\end{proof}

\item  \label{orthogonal complement of a type $U^{IJ}$ subspace is of type $U^{IJ}$} 

The orthogonal sum of subspaces of $U$ of  type $U^{IJ}$ 
  is a subspace of type  $U^{IJ}$.
Furthermore, let $W_1 \subset U^{2m}$ be  a  subspaces  of type $U^{IJ}$.  The subspaces   $W_1^\perp \cap U$   is of  type $U^{IJ}$.

\begin{proof}
 Straightforward.
\end{proof}

\item
  Let $U \in \mathcal{IC}^{2m}$ be not 2-planes decomposable. 
  If a pair of 4-dimensional subspaces $U_1,U_2$ of $U$  of type  $U^{IJ}$ 
has a non trivial intersection then $U_1=U_2$.
As a consequence for any $ U \ni \bar X \notin U_1$  one has  $U_1 \cap U^{IJ}(\bar X)= \{ 0\}$.
 In particular if  $\bar X \in U_1^\perp$ all vectors of the $\omega^I$ and $\omega^J$ chains centered on $\bar X$ are in $U_1^\perp$.
\begin{proof}
The first statement is straightforward from  point (\ref{invariance of $U^{IJ}(X_1)$ from the leading vector X di $U^{IJ}(X_1)$}). For the second,
let $\bar X_1 \in U_1^\perp$. Without lack of generality we can suppose that both $\cos \theta^I$ and $\cos \theta^J$ are not zero. Then $\bar X_2 =\frac{I^{-1}P^{IU}\bar X}{\cos \theta^I}  \in U_1^\perp$ since, from previous point,  $U_1^\perp $ is of type $U^{IJ}$  and for the same reason also $\bar Y_2=\frac{J^{-1}P^{JU}\bar X}{\cos \theta^J}  \in U_1^\perp$. Then the 2-plane $L(\bar X_2,\bar Y_2) \subset U_1^\perp$ which implies that the pair $(\bar X_4, \bar Y_4)$ is in  $U_1^\perp$ as well as  $\bar X_3= \bar Y_3$. Then $L(\bar X_1,\bar X_2) \stackrel{\perp} {\oplus} L(\bar X_3,\bar X_4)\subset  U_1^\perp$.
\end{proof}

\item
Let $U^{IJ}(X_1)=L(X_1,X_2,X_3,X_4)$ and $U^{IK}(X_1)=L(X_1,X_2,\tilde X_3,\tilde X_4)$   be a pair  of associated  subspaces. Then either  $U^{IJ}(X_1) \cap U^{IK}(X_1)= L(X_1,X_2)$ or $U^{IJ}(X_1)= U^{IK}(X_1)$. 
In the last case   $U^{IJ}(X_1)=U^{IK}(X_1) \in \mathcal{IC}^4$ (with same angles of isoclinicity of $U$).
\begin{proof}
In case  at least one  among $\xi,\chi,\eta$ equals $\pm 1$, 
from the Proposition (\ref{orthogonal subspaces belong to mathcal{IC} 4 with angles theta I,theta J,theta K}), one has that $U^{IJ}(X_1)=U^{IK}(X_1)=U^{JK}(X_1)$.
Otherwise 
the 4-dimensional subspaces $U^{IJ}(X_1)$ and $U^{IK}(X_1)$  intersect by construction in $L(X_1,X_2)$ since $L(X_1,X_2)$ is a standard 2-plane of $\omega^I$ in both of them.
Suppose $S \in U^{IJ}(X_1) \cap U^{IK}(X_1)$ with $S \notin L(X_1,X_2)$. Then $S= S_1 + S_2 $ with $S_1 \in   L(X_1,X_2)$ and $S_2 \in L(X_1,X_2)^\perp=L(X_3,X_4)$ as a vector of $U^{IJ}(X_1)$ and
$S= T_1 + T_2 $ with $T_1 \in   L(X_1,X_2)$ and $T_2 \in L(X_1,X_2)^\perp= L(\tilde X_3, \tilde X_4)$ as a vector of $U^{IK}(X_1)$.  Being  $S_1=T_1$, by the uniqueness of the orthogonal  decomposition,  it follows  $S_2=T_2$. Then $L(\tilde X_3, \tilde X_4)\cap L(X_3,X_4) \neq \{0\}$. which implies that  $L(\tilde X_3, \tilde X_4)= L(X_3,X_4)$ i.e. $U^{IJ}(X_1)= U^{IK}(X_1)$.
\end{proof}

\item \label{1:1 correspondence between standard subspaces of $U^{IJ}(X_1)$ and associated subspaces of type $U^{IK}$}
  Let $X_1 \in U$ unitary and $U^{IJ}(X_1)$ and $U^{IK}(X_1)$ be a pair of associated 4-dimensional subspaces.
   Then either $U^{IJ}(X_1)=U^{IK}(X_1)$ i.e. $U^{IJ}(X_1) \in \mathcal{IC}^{4}$ with same angles of isoclinicity of $U$ or there exists a $1 : 1$ correspondence between standard 2-planes of $\omega^I|_{U^{IJ}(X_1)}$ and associated subspaces of type $U^{IK}$.
\begin{proof}
From previous point  either  $U^{IJ}(X_1) \cap U^{IK}(X_1)= L(X_1,X_2)$ or $U^{IJ}(X_1)= U^{IK}(X_1)$ in which case $U^{IJ}(X_1) \in \mathcal{IC}^4$ (with same angles of isoclinicity of $U$).
Suppose  $U^{IJ}(X_1) \notin \mathcal{IC}^4$. Let  $\bar X_1 \in U^{IJ}(X_1), \; \bar X_1 \notin L(X_1,X_2)$.   Clearly $U^{IK}(\bar X_1) \neq U^{IJ}(\bar X_1)=U^{IJ}(X_1)$ by the assumption that  $U^{IJ}(X_1) \notin \mathcal{IC}^4$. We prove that $U^{IK}(X_1) \neq U^{IK}( \bar X_1)$. In fact if $U^{IK}(X_1) = U^{IK}( \bar X_1)$ one has that $ \bar X_1 \in U^{IK}(X_1)$ and consequently  by previous point $U^{IK}(X_1)=U^{IJ}(X_1)$ contradiction.

\end{proof}

\end{enumerate}
\end{prop}

Let $X_1 \in U$  unitary  and $U^{IJ}(X_1)$ and $U^{IK}(X_1)$ be the associated subspaces.
It is straightforward to verify that the  Proposition (\ref{orthogonal maps between chains belong to Sp(n)})
and the Proposition (\ref{invariance of $C_{IJ}$ and $C_{IK}$ w.r.t. any leading vector}) can be applied also to the 4-dimensional subspaces of type $U^{IJ},U^{IK},U^{JK}$ leading to the
invariance of the  Gram matrices $g_{IJ}=<X_i,Y_j>$ and $g_{IK} <X_i,Z_j>$ w.r.t. the chains $\{X_i\}, \{Y_i\},\{Z_i\}$ centered on $X_1$. The form of such matrices
 is still given in
   (\ref{canonical matrices $C_{IJ}$ $C_{IK}$ and of 4-dimensional isoclinic subspace w.r.t. the associated chains}).
In this case the subspaces $L(X_3,X_4)$ and $L(\tilde X_3,\tilde X_4)$ are in general different as well as the pairs of associated  subspaces $U^{IJ}(X_1)$ and $U^{IK}(X_1)$. Therefore
 $g_{IK}$ is not in general an orthogonal matrices (differently from $g_{IJ}$).   
 Recalling that from the Proposition (\ref{Gamma is an invariant of an isoclinic 4 dimensional subspace 1})  one has that  $\Gamma= <X_3,\tilde X_3>$ is   given in (\ref{cos phi}) in case none  among $\xi,\chi,\eta$ equals  $\pm 1$   and $\Gamma=1$  otherwise we state the following
\begin{prop} \label{values of gamma and delta}
The Gram matrix  $(L(X_3,X_4) \times L(\tilde X_3,\tilde X_4))$ is given by
\[
 \left(   \begin{array}{cc}
 <X_3,\tilde X_3> & <X_3,\tilde X_4>\\
<X_4,\tilde X_3> & <X_4,\tilde X_4>
\end{array}      \right) =
 \left(   \begin{array}{cc}
 \Gamma & -\Delta\\
\Delta  & \Gamma \end{array}      \right).
\]
  Therefore the 2-planes $L(X_3,X_4)$ and $L(\tilde X_3,\tilde X_4)$ are  isoclinic with the cosine of the angle of isoclinicity $\theta$ given by $\cos \theta= \sqrt{\Gamma^2+\Delta^2}$.
\end{prop}

\begin{proof}

 We prove that $<X_3,\tilde X_4>=-<X_4,\tilde X_3>$. In fact
\[
\begin{array}{lll}
<X_3,\tilde X_4>&=&<IX_3,I\tilde X_4>=- <\frac{P^{IU} X_4}{\cos \theta^I},I \tilde X_4>=  -\frac{1}{\cos \theta^I}<X_4, P^{IU} I \tilde X_4>=\\
       &=& -\frac{1}{\cos \theta^I}<X_4, I\tilde X_4>=   \frac{1}{\cos \theta^I}<IX_4, \tilde X_4>=\frac{1}{\cos \theta^I}<P^{IU} IX_4, \tilde X_4>=\frac{1}{\cos \theta^I}<IX_4, P^{IU}\tilde X_4>=\\
       &=& \frac{1}{\cos \theta^I}<X_4, I^{-1}P^{IU}\tilde X_4>=\frac{1}{\cos \theta^I}<X_4, -\tilde X_3 \cos \theta^I>=-<X_4,\tilde X_3>
\end{array}
\]

Let now compute
\[
<X_4,\tilde X_4>=  <\frac{Y_2- \xi X_2}{\sqrt{1-\xi^2}},\frac{Z_2- \chi X_2}{\sqrt{1-\chi^2}}>= \frac{<Y_2- \xi X_2,Z_2- \chi X_2>}{\sqrt{1-\xi^2}\sqrt{1-\chi^2}}=
 \frac{\eta- \chi \xi -\xi \chi + \xi \chi}{\sqrt{1-\xi^2}\sqrt{1-\chi^2}}= \frac{\eta- \chi \xi} {\sqrt{1-\xi^2} \sqrt{1-\chi^2}}
\]
\end{proof}

It is easy to check that $\Delta=<X_4,\tilde X_3>$ 
is given by the following equivalent expressions where $\epsilon= -\frac{1}{\cos \theta^I \sqrt{1-\xi^2} \sqrt{1-\chi^2}}$:     
\begin{equation} \label{equivalent exprssions of Delta}
\begin{array}{lll}
\Delta=<X_4,\tilde X_3> &=& \frac{1}{\cos \theta^I}<X_4,I\tilde X_4>=
 -\frac{1}{\cos \theta^I}<X_3,I\tilde X_3>=
 -\epsilon <Y_2,I\tilde X_4>=
  \frac{\epsilon}{\cos \theta^J  \cos \theta^K} <P^{JU} X_1,P^{KU} X_1>\\
&=&  \frac{<Z_2,X_3>}{\sqrt{1-\chi^2}}=
 \frac{1}{\cos^2 A}  <P^{IU} X_4, P^{IU} \tilde X_3>=
  \frac{1}{\cos^2 A}   <P^{IU} \tilde X_4, P^{IU}  X_3>=
 -\epsilon <Y_2,IZ_2>.
\end{array}
\end{equation}

The  Proposition (\ref{orthogonal maps between chains belong to Sp(n)})
and the Proposition (\ref{invariance of $C_{IJ}$ and $C_{IK}$ w.r.t. any leading vector})
lead to the
following  strong characterization of the subspaces of type $U^{IJ}$. 
\begin{prop} \label{All 4 dimensional subspaces of type $U^{IJ}$  have the same value of Gamma e Delta}
Let $U \in \mathcal{IC}^{2m}$. All 4-dimensional subspaces of $U$ of type $U^{IJ}$  have the same value of the pair $(\Gamma, \Delta)$ (w.r.t. the associated subspaces of type $U^{IK}$).
\end{prop}
In the Proposition (\ref{properties of subspaces of tupe $U^{IJ}$, $U^{IK}$, $U^{JK}$}) we stated that given a unitary vector $X_1 \in U$ and determined the associated pair $U^{IJ}(X_1),U^{IK}(X_1)$, for any vector $ \bar X \in U^{IJ}(X_1)$ one has that $U^{IJ}( \bar X)=U^{IJ}(X_1)$  (point (\ref{invariance of $U^{IJ}(X_1)$ from the leading vector X di $U^{IJ}(X_1)$})) whereas, unless $ \bar X \notin L(X_1,X_2)$, it is   $U^{IK}(\bar X) \neq U^{IK}(X_1)$ (point (\ref{1:1 correspondence between standard subspaces of $U^{IJ}(X_1)$ and associated subspaces of type $U^{IK}$})). Nevertheless, from the Proposition (\ref{All 4 dimensional subspaces of type $U^{IJ}$  have the same value of Gamma e Delta})  it follows the interesting
\begin{coro} \label{Principal angles of all pairs of subpspaces $(U^{IJ}(X_1),U_i^{IK})$}
 Let $U^{IJ}(X_1) \subset U$ be  a 4-dimensional subspace of type $U^{IJ}$. Suppose that $U^{IJ}(X_1) \neq U^{IK}(X_1)$  and denote by $U_i^{IK}$ the different  subspaces associated to $U^{IJ}(X_1)$ determined by different unitary vectors of $U^{IJ}(X_1)$. The  cosines of the principal angles between the pairs  of associated subspaces $(U^{IJ}(X_1),U_i^{IK})$  are always given by $(1,1,\sqrt{ \Gamma^2 + \Delta^2},\sqrt{ \Gamma^2 + \Delta^2})$.
\end{coro}

Applying  the  Proposition (\ref{orthogonal maps between chains belong to Sp(n)}) we extend the result of the Proposition (\ref{canonical matrices of any  subspace of ${IC}^(4)$}) to state the following
\begin{prop} \label{independence of canonical matrices from the leading vector}
Let $\{X_i \},\{Y_i \},\{Z_i \}$ be the  triple of $\omega^I,\omega^J,\omega^K$-chains centered on  $X_1 \in U$.
 The Gram matrices  $g_{IJ}= <X_i,Y_j>$ and $g_{IK}= <X_i,Z_j>$  are invariants of $U$.
\end{prop}
We recall that the form of these matrices is given in
(\ref{canonical matrices $C_{IJ}$ $C_{IK}$ and of 4-dimensional isoclinic subspace w.r.t. the associated chains})
and  that  in general  $g_{IK}$, differently from $g_{IJ}$,  is not an orthogonal matrix.

Let $U \in \mathcal{IC}^{2m}$  
 and let $U^{IJ}(X_1)=L(X_1,X_2,X_3,X_4)$ and $U^{IK}(X_1)=\{\tilde X_i\}=(X_1,X_2,\tilde X_3,\tilde X_4)$  be the  pair of associated subspaces centered on $X_1$.
  \begin{prop} \label{The pair $(U^{IJ}(X_1),KU^{IJ})$ is isoclinic.}
Chosen $X_1 \in U$, the  4-dimensional subspace $U^{IJ}(X_1)$  is isoclinic with $KU^{IJ}(X_1)$. Then for any leading vector $X_1$ all subspaces   $U^{IJ}(X_1)  \in \mathcal{IC}^4$ with angles  of isoclinicity  $(\theta^I,\theta^J,\gamma)$  where 
\[ \cos^2 \gamma= <X_1,KX_2>^2 + <X_1,KX_3>^2 + <X_1,KX_4>^2= \cos^2 \theta^K (\Gamma^2 + \Delta^2 + \chi^2(1-\Gamma^2-\Delta^2)).\]
All 4-dimensional subspaces of $U$ of type $U^{IJ}$ share the same  matrix of $\omega^K$ w.r.t.  the generating chain.
\end{prop}
\begin{proof}
If some among $\xi,\chi,\eta$ equals $\pm 1$
from the Proposition (\ref{orthogonal subspaces belong to mathcal{IC} 4 with angles theta I,theta J,theta K}), $U^{IJ}(X_1)=U^{IK}(X_1)$ is a 4-dimensional isoclinic subspace with same angles $(\theta^I,\theta^J, \theta^K)$ of $U$.
 Consider then the case that  none among $\xi,\chi,\eta$ equals $\pm 1$ (in particular the case that $U$ is not orthogonal). 
 From  Corollary (\ref{Principal angles of all pairs of subpspaces $(U^{IJ}(X_1),U_i^{IK})$})  we have that the  cosines of the principal angles of the pair $(U^{IJ}(X_1),U^{IK}(X_1))$  are given by $(1,1,\sqrt{\Gamma^2 + \Delta^2},\sqrt{\Gamma^2 + \Delta^2})$ with $\gamma$  given in (\ref{cos phi})
  and
\[
\Delta= <X_4,\tilde X_3>=-<X_3,\tilde X_4>= - <X_3, \frac{Z_2-\chi X_2}{\sqrt{1- \chi^2}}>= - \frac{<X_3, K^{-1}P^{KU} X_1>}{\cos \theta^K \sqrt{1- \chi^2}}=  \frac{<X_1,K X_3>}{\cos \theta^K \sqrt{1- \chi^2}}.
\]
This is another expression for $\Delta$ besides the ones given in (\ref{equivalent exprssions of Delta}).  We have then that the cosine $<X_1,KX_3>$ is an invariant of $U$.


Let   consider  a new  leading vector $ \bar X_1 \in U^{IJ}(X_1)$ which in the general case will be $\bar X_1= a X_1 + b X_2 + c X_3 + d X_4, \quad a^2+ b^2 + c^2 + d^2=1$.
 One has
\[
\begin{array}{l}
\bar X_2= \frac{I^{-1}P^{IU} \bar X_1}{\cos \theta^I}=  (a X_2  -b X_1 + c X_4 - d X_3); \\
\bar Y_2= \frac{J^{-1}P^{JU} \bar X_1}{\cos \theta^J}=  (a Y_2  +b \frac{J^{-1}P^{JU}  X_2}{\cos \theta^J} + c Y_4 +d   \frac{J^{-1}P^{JU}  X_4}{\cos \theta^J}).
\end{array}
\]

Let compute $\frac{J^{-1}P^{JU}  X_2}{\cos \theta^J}$ and $\frac{J^{-1}P^{JU}  X_4}{\cos \theta^J}$. In order that $L(X_2,X_4)$ and $L(Y_2,Y_4)$ have  the same orientation it is:
\[
X_2=\xi Y_2  -\sqrt{1-\xi^2} Y_4, \qquad  \qquad X_4=\sqrt{1-\xi^2} Y_2 + \xi Y_4.
\]

Then computing
\[
\frac{J^{-1}P^{JU}  X_2}{\cos \theta^J}= - \xi   X_1 +  \sqrt{1-\xi^2}  X_3, \qquad   \qquad \frac{J^{-1}P^{JU}  X_4}{\cos \theta^J}= -\sqrt{1-\xi^2}  X_1 - \xi  X_3
\]

 Then
 \[
\begin{array}{lll}
\bar Y_2 &=&   a Y_2 + b(-\xi X_1 + \sqrt{1-\xi^2} X_3) + c Y_4 + d(-\sqrt{1-\xi^2} X_1 - \xi X_3)=\\
\end{array}
\]
and  being $Y_2=\xi X_2 +  \sqrt{1-\xi^2} X_4$ and $Y_4=-\sqrt{1-\xi^2} X_2 + \xi X_4$ one has

\[
\bar Y_2=(-b \xi - d \sqrt{1-\xi^2})X_1 + (a \xi - c \sqrt{1-\xi^2})X_2 + (b \sqrt{1-\xi^2} - d \xi) X_3 + (a \sqrt{1-\xi^2} +c \xi)X_4.
\]

One can verify that $|\bar Y_2|=1$ and that $<\bar X_2,\bar Y_2>=\xi$. 
Being $\bar X_4=  \frac{\bar Y_2- \xi \bar X_2}{\sqrt{1-\xi^2}}$,  one has
\[\bar X_4=-d X_1 - c X_2 + b X_3 + a X_4\] and
\[\bar X_3=   - \frac{I^{-1}P^{IU}\bar X_4}{\cos \theta^I}= -c X_1+ d X_2 + a X_3 - b X_4.\]

Through simple calculations one has

\[
\begin{array} {lll}
<\bar X_1,K \bar X_3>  & =&    (ad+bc)[<X_1,KX_2> - <X_3,KX_4>]  + (ab-cd) [<X_2,KX_3> - <X_1,KX_4>]\\
& = & + (a^2 + c^2) <X_1,KX_3> - (b^2 +d^2)<X_2,KX_4>.
\end{array}
\]

This is the value of $<\bar X_1,K \bar X_3>=<X_1,K X_3> $ for any chain centered  on a unitary  vector  $\bar X_1 \in U^{IJ}(X_1)$. In particular such expression must be valid in case $\bar X_1 \in L(X_1,X_2)$ i.e. in case  $c=d=0$.
In such case, it is
\[
<\bar X_1,K \bar X_3>=     +a^2 <X_1,KX_3> -b^2 <X_2,KX_4>+ab [<X_2,KX_3> - <X_1,KX_4>]
\]
 which implies
\[
<X_1,KX_4> = <X_2,KX_3>, \qquad <X_1,KX_3> =-  <X_2,KX_4>.
\]

The matrix of $\omega^K|_{U^{IJ}}$ w.r.t. the $\omega^I$-chain  $(X_1,X_2,X_3,X_4)$  has then the form  (\ref{matrix of projector omega^A of an isoclinic 4 dimensional subspace})
 and the conclusion follows from  the Proposition (\ref{general form for the matrix of projector omega^A of an isoclinic 4 dimensional subspace}).
Namely, computing, one has
\[
\omega^K|_U=\left(
\begin{array} {cccc}
0  & \chi \cos \theta^K &  \Delta  \cos \theta^K \sqrt{1-\chi^2} &  \Gamma \cos \theta^K \sqrt{1-\chi^2}\\
- \chi \cos \theta^K & 0 & \Gamma \cos \theta^K \sqrt{1-\chi^2} & -\Delta  \cos \theta^K \sqrt{1-\chi^2}\\
-\Delta  \cos \theta^K \sqrt{1-\chi^2} & -\Gamma \cos \theta^K \sqrt{1-\chi^2} & 0 &\chi \cos \theta^K\\
-\Gamma \cos \theta^K \sqrt{1-\chi^2} &  \Delta  \cos \theta^K \sqrt{1-\chi^2} & -\chi \cos \theta^K & 0
\end{array}
\right).
\]
which ends the proof.
\end{proof}

Observe that, if $U$ is orthogonal, 
we affirmed beforehand that  we can assume $\Gamma=1$ (and $\Delta=0$). In this case we have a confirmation that $\gamma= \cos \theta^K$ and $U \in \mathcal{IC}^4$ with angles $(\theta^I, \theta^J,\theta^K)$. The same conclusion is valid also if  $\Delta=-1$ which happens for instance in case of 4-dimensional quaternionic and $I$-complex subspaces (see examples in previous section)  and in general if $\Gamma^2 + \Delta^2=1$ i.e. if $U^{IJ}(X_1)$ is of type $U^{IK}$ as well.
In fact one has the
\begin{coro} \label{cases when gamma= cos theta K}
Let $U$ be a 4-dimensional subspace of type $U^{IJ}$.   Then  $\gamma= \cos \theta^K$  iff  $\Gamma^2 + \Delta^2=1$. In this case, for any   $X \in U$, $U^{IJ}(X)=U^{IK}(X)=U$.
\end{coro}
\begin{proof}
The solutions of  $(\Gamma^2 + \Delta^2 + \chi^2(1-\Gamma^2-\Delta^2))=1$ are $\Gamma^2 + \Delta^2=1$ and $\chi=\pm 1$. But in the second case, from the  Proposition (\ref{Gamma is an invariant of an isoclinic 4 dimensional subspace 1}), it is  $\Gamma=1$. The second statement follows from  the property (\ref{invariance of $U^{IJ}(X_1)$ from the leading vector X di $U^{IJ}(X_1)$}) of the Proposition (\ref{properties of subspaces of tupe $U^{IJ}$, $U^{IK}$, $U^{JK}$})  and  from the hypothesis that $\tilde X_3 \in L(X_3,X_4)$.
\end{proof}



\subsection{Orthogonal decomposition of  $U \in \mathcal{IC}^{2m}$  ($m\geq 4$) into isoclinic addends.} 

In this section we show that any $U \in \mathcal{IC}^{2m}$ admits an orthogonal decomposition into isoclinic addends $U_i$. Although such decomposition is in general not unique, the addends have all the same dimension and, for any $A \in S(\mathcal{Q})$, the  angle of isoclinicity of all pair $(U_i,AU_i)$  is the same as the one of the pair $(U,AU)$. Fixed an admissible basis $(I,J,K)$ and  being $(\theta^I,\theta^J,\theta^K)$ the respective angles of isoclinicity,  the addends are  characterized by the same values of a pair of invariants $(\Gamma,\Delta)$. It will turn out that  the angles $(\theta^I,\theta^J,\theta^K)$  together with  $(\xi,\chi,\eta, \Delta)$  determines the $Sp(n)$-orbit of $U$. We first prove the following
\begin{teor} \label{any X determines an 8 dimensional addend of an isoclinic subspace}
Let $U \in \mathcal{IC}^{2m}, \; m>4$   and, for any $A \in S(\mathcal{Q})$, denote by $\theta^A$ the angle of isoclinicity of the pair $(U,AU)$.
For any unitary $X_1 \in U $ there exists  an 8-dimensional subspace   $U^8 \in \mathcal{IC}^8$ containing $X_1$ such that,  for any $A \in S(\mathcal{Q})$, the pair  $(U^8,AU^8)$ has angle  $\theta^A$.
The subspace $U^8$ is unique unless  it is furtherly decomposable into isoclinic subspaces with same angle $\theta^A$  for any $A \in S(\mathcal{Q})$.
\end{teor}

\begin{proof}


In case $U$ is a non orthogonal 2-planes decomposable subspace  
any $\omega^I$-standard 2-plane     is isoclinic with angles $(\theta^I,\theta^J,\theta^K)$.
If $U$ is double orthogonal (resp. triple orthogonal),
any standard 2-plane associated to the non null cosine (resp. any 2-plane) is a standard 2-plane for $\omega^I,\omega^J,\omega^K$.
 Then $U^8$ is obtained by summing up  four mutually orthogonal  such standard 2-planes (the first containing $X_1$).  Clearly $U^8$ is not uniquely determined.
Observe that the orthogonal sum of a pair of such 2-plane is a 4-dimensional isoclinic subspace  with angles $(\theta^I,\theta^J,\theta^K)$ characterized then by $\Gamma=1$ and $\Delta=0$.

Consider now the case that $U$ is not a 2-planes decomposable subspace.
Let $U^{IJ}(X_1)=L(X_1,X_2,X_3,X_4)$ and $U^{IK}(X_1)=L(X_1,X_2,\tilde X_3,\tilde X_4)$ be the pair of  (unique) associated subspaces centered on $X_1$. Let $\Gamma$ and $\Delta$ as in the Proposition (\ref{values of gamma and delta}) and expression   (\ref{equivalent exprssions of Delta}). From the  Proposition (\ref{All 4 dimensional subspaces of type $U^{IJ}$  have the same value of Gamma e Delta}) the pair $(\Gamma,\Delta)$ is an invariant of all 4-dimensional subspaces of type $U^{IJ}$ of $U$. 

In the sequel, for simplicity we will not indicate the leading vector and denote by $(U^{IJ})^\perp$ (resp. $(U^{IK})^\perp$)  the subspace $(U^{IJ}(X_1))^\perp \cap U$ (resp.  the subspace $(U^{IK}(X_1))^\perp \cap U$).
From the Proposition (\ref{properties of subspaces of tupe $U^{IJ}$, $U^{IK}$, $U^{JK}$}), the subspace  $(U^{IJ})^\perp$ (resp. $(U^{IK})^\perp$) is of  type $U^{IJ}$ (resp. of  type $U^{IK}$).

 The subspaces   $(U^{IJ})^\perp $ and $(U^{IK})^\perp $ are in general different unless
 $\Gamma^2 + \Delta^2=1$.


As stated in Corollary (\ref{cases when gamma= cos theta K}), in  case  $\Gamma^2 + \Delta^2 =1$  the subspaces $U^{IJ}=U^{IK} \in \mathcal{IC}^4$ with angles $(\theta^I,\theta^J,\theta^K)$ w.r.t. $(I,J,K)$.
In this case, 
considering  any unitary vector $X_5 \in (U^{IJ})^\perp$, built the  $\omega^I$-chains  $\{X_i\},\{\tilde X_i\}$  centered on $X_5$ and being  $U^{IJ}(X_5)$ and $U^{IK}(X_5)$  the associated subspaces of type $U^{IJ}$ and $U^{IK}$,  one has that $U^{IJ}(X_5)=U^{IK}(X_5) \in \mathcal{IC}^4$ with angles $(\theta^I,\theta^J,\theta^K)$ w.r.t. $(I,J,K)$. Furthermore from point (\ref{orthogonal complement of a type $U^{IJ}$ subspace is of type $U^{IJ}$}) of the Proposition (\ref{properties of subspaces of tupe $U^{IJ}$, $U^{IK}$, $U^{JK}$}) one has that $U^{IJ}(X_1) \perp U^{IJ}(X_5)$.

The 8-dimensional subspace  $U_1^8=U^{IJ}(X_1) \stackrel{\perp}  \oplus U^{IJ}(X_5)= U^{IK}(X_1) \stackrel{\perp}  \oplus U^{IK}(X_5)$ is isoclinic with $IU^8,JU^8,KU^8$ with  angles  $(\theta^I,\theta^J,\theta^K)$. Because of the arbitrariness in  the choice of $X_5$, also in this case $U^8$ is not uniquely defined.


Let then consider the case that $\Gamma^2 + \Delta^2 \neq 1$.   From the Proposition (\ref{propertiers of CS decomposition}),  the cosines of the pair of  non null  principal angles  of the pair $((U^{IJ})^\perp, (U^{IK})^\perp) $ are both  equal to $\sqrt{\Gamma^2 + \Delta^2}$.    Let $U_1 \subset (U^{IJ})^\perp$ and $U_2 \subset (U^{IK})^\perp$ be the  pair of isoclinic 2-planes generated by the related principal vectors associated to  such non null principal angles. One has
\[
(U^{IJ})^\perp = U_1  \stackrel{\perp}{\oplus} W; \qquad \qquad  (U^{IK})^\perp = U_2  \stackrel{\perp} {\oplus} W
\]
being $W$ the $(2m-6)$-dimension intersection $(U^{IJ})^\perp   \cap (U^{IK})^\perp$.
\begin{lemma}
The subspaces $U_1$ and $U_2$ are  $\omega^I$-standard.
\end{lemma}
\begin{proof}
Both  $(U^{IJ})^\perp $ and $(U^{IK})^\perp $ are isoclinic with  $I(U^{IJ})^\perp $ and $I(U^{IK})^\perp$ with angle $\theta^I$. Then for any $X_1 \in (U^{IJ})^\perp$ the vector  $X_2= \frac{I^{-1}P^{IU}X_1}{\cos \theta^I}$ belongs to $(U^{IJ})^\perp$ and analogously  for any $\bar {X_1} \in (U^{IK})^\perp $ the vector   $\bar X_2=\frac{I^{-1}P^{IU} \bar X_1}{\cos \theta^I}$  is in   $(U^{IK})^\perp  $. Consequently any $X_1 \in (U^{IJ})^\perp \cap (U^{IK})^\perp=W$ has the corresponding in $W$ which implies that  both $U_1$ and $U_2$ are $\omega^I$-standard 2-planes.
\end{proof}
It follows that
\[
Pr^{IU} U_1= IU_1, \qquad Pr^{IU} U_2= IU_2
\]
and $Pr^{IU}W=IW$.

 Consider an unitary vector  $X_7 \in U_1$ and  the $\omega^I$-chain $(X_7,X_8,X_5,X_6)$ generating the subspace   $U^{IJ}(X_7)= L(X_7,X_8) \stackrel{\perp}{\oplus} L(X_5,X_6)$. From previous Lemma $U_1= L(X_7,X_8)$. From the Proposition (\ref{properties of subspaces of tupe $U^{IJ}$, $U^{IK}$, $U^{JK}$}) point 4), $U^{IJ}(X_7) \subset U_1^{IJ}(X_1)^\perp$.
The $\omega^I$-standard 2-plane $L(X_5,X_6) \subset W$.
In fact, from previous Lemma, $U^{IJ}(X_7)= L(X_7,X_8) \stackrel{\perp}{\oplus} L(X_5,X_6)\subset {U_1^{IJ}}^\perp= L(X_7,X_8) \stackrel{\perp}{\oplus} W$ and, being  $L(X_5,X_6) \perp L(X_7,X_8)$, the conclusion follows.

Let consider now the  vector $X_5$ as new leading vector. From the Proposition (\ref{properties of subspaces of tupe $U^{IJ}$, $U^{IK}$, $U^{JK}$}), one has that $U^{IJ}(X_5)= U^{IJ}(X_7)$. In fact the $\omega^I$-chain is now  $(X_5,X_6, -X_7,-X_8)$  and  $U^{IJ}(X_5)= L(X_5,X_6) \stackrel{\perp}{\oplus}  L(X_7,X_8)$. Let  $U^{IK}(X_5)= L(X_5,X_6) \stackrel{\perp}{\oplus}  L(-\tilde X_7,-\tilde X_8)$ the associated subspace of type $U^{IK}$ where the generators form the $\omega^I$-chain centered on $X_5$. From point 4) of the Proposition (\ref{properties of subspaces of tupe $U^{IJ}$, $U^{IK}$, $U^{JK}$}),
the chain $U^{IK}(X_5) \subset (U^{IK})^\perp=U_2 \stackrel{\perp}{\oplus} W$.
\begin{prop}
One has
\[
L(\tilde X_7,\tilde X_8)=U_2
\]
\end{prop}

\begin{proof}
By assumption, the orthogonal projection of $L(X_7,X_8)$ onto 
  $(U^{IK})^\perp=U_2 \oplus W$ 
   is onto $U_2$ 
    with  the cosine of the principal angles  equal to $\sqrt{\Gamma^2 + \Delta^2}$. From the Proposition (\ref{All 4 dimensional subspaces of type $U^{IJ}$  have the same value of Gamma e Delta}) as 2-planes of the associates chains $(X_5,X_6,-X_7,-X_8)$  and $(X_5,X_6,-\tilde X_7,-\tilde X_8)$ the subspaces   $L(X_7,X_8)$ and  $L(\tilde X_7,\tilde X_8)$ are isoclinic with the cosine of the principal angles  equal to $\sqrt{\Gamma^2 + \Delta^2}$. Then $L(\tilde X_7, \tilde X_8)=U_2$.
\end{proof}

Therefore $L(X_3,X_4, X_7,X_8)=L(\tilde X_3,\tilde X_4,\tilde X_7,\tilde X_8)$ and
the 8-dimensional subspace
 \[U_1^8=L(X_1,X_2,X_3,X_4,X_5,X_6,X_7,X_8)=L(X_1,X_2,\tilde X_3,\tilde X_4,X_5,X_6,\tilde X_7,\tilde X_8)\]
 is orthogonal sum of a  pair of uniquely defined subspaces of type $U^{IJ}$ that is $U_1^8=U^{IJ}(X_1) \oplus U^{IJ}(X_5)$ and  orthogonal sum of the pair of the (uniquely defined) relative associated subspaces of type $U^{IK}$ that is  $U_1^8=U^{IK}(X_1) \oplus U^{IK}(X_5)$.
This implies that $U_1^8$ is isoclinic with $IU_1^8,JU_1^8,KU_1^8$ with angles $(\theta^I,\theta^J,\theta^K)$. By the uniqueness of the addends it follows the uniqueness of $U^8$.
To conclude the proof we need the following
\begin{lemma} \label{the 8-dimensional addend is isoclinic}
Let $U \in \mathcal{IC}^{2m}$, $(I,J,K)$ be some admissible basis and $(\theta^I,\theta^J,\theta^K)$ be the angles of isoclinicity of the pairs $(U,IU), (U,JU),(U,KU)$ respectively. Any $k$-dimensional subspace $W \subset U$  isoclinic with $IW,JW,KW$ with  angles  equal respectively to $(\theta^I,\theta^J,\theta^K)$   belongs to $\mathcal{IC}^{k}$ and, for any $A \in S(\mathcal{Q})$,   the angle of isoclinicity of the pair   $(W,AW)$ equals the one of the pair $(U,AU)$. 
\end{lemma}
\begin{proof}
 Let $A= \alpha I + \beta J + \gamma K$ be a compatible complex structure and let $\theta^A$  be the angle of isoclinicity of the pair $(U,AU)$. Consider the subspace  $AW=\{AX=\alpha IX + \beta JX + \gamma KX, \, X \in W \}$. If $\theta^A = \pi/2$ 
 the pair $(W,AW)$ is isoclinic with angle $\pi/2$. Suppose then  that  $\theta^A \neq \pi/2$. 
 We prove that $Pr^{AU}W= AW$ which implies that the pair  $(W,AW)$ is isoclinic  with angle of isoclinicity equal to $\theta^A$. Projecting the generic vector
    $AW \ni \tilde Y= \alpha I \tilde X + \beta J \tilde X + \gamma K \tilde X$ with $\tilde X \in W$ onto $U$ one has
    \[
    Pr^U \tilde Y= Pr^U  (\alpha I \tilde X + \beta J \tilde X + \gamma K \tilde X) = \alpha Pr^U I \tilde X  + \beta Pr^U J \tilde X + \gamma Pr^U K \tilde X.
    \]

    By hypothesis the pairs $(W,IW),(W,JW),(W,KW)$ are isoclinic with angles $\theta^I,\theta^J,\theta^K$ respectively  which implies 
    that $Pr^{IU}W=IW$ and $Pr^{U}IW=W$, $Pr^{JU}W=JW$  and $Pr^{U}JW=W$, $Pr^{KU}W=KW $and  $Pr^{U}KW=W$. Therefore $Pr^U \tilde Y \in W$. This implies that the pair $(W,AW)$ is isoclinic with  the angle of isoclinicity equal to $\theta^A$.
\end{proof}
\end{proof}

\begin{lemma}
For any $X \in U_1^8$  the associated subspaces  $U^{IJ}(X)$ and $U^{IK}(X)$ are both in $U_1^8$.
\end{lemma}
\begin{proof}
For any $X_1 \in U_1^8$ one has that $X_2,Y_2,Z_2$ completing the standard 2-planes of $\omega^I,\omega^J,\omega^K$ respectively are in $U_1^8$ (by assumption in case of triple orthogonality)   and so are the 2-plane $L(X_2,Y_2),L(X_2,Z_2),L(Y_2,Z_2)$ and consequently $X_4,Y_4,Z_4$ and the whole chains $\{X_i\},\{Y_i\},\{Z_i\}$.
\end{proof}


Let now consider the 4-dimensional $\omega^I$-standard subspace $L(X_3,X_4, X_7,X_8)=L(\tilde X_3,\tilde X_4,\tilde X_7,\tilde X_8)$.
The Gram matrix w.r.t. the bases  $(X_3,X_4,-X_7,-X_8)$ and $(\tilde X_3,\tilde X_4,-\tilde X_7,-\tilde X_8)$ has necessarily the form
\begin{equation} \label{matrix of a 4 plane w.r.t. the associated chains}
A=\left(
\begin{array} {cc|cc}
 \Gamma & -\Delta &  a & b   \\
 \Delta & \Gamma & b & -a \\
\hline
-a & -b & \Gamma & -\Delta \\
 -b & a & \Delta & \Gamma
\end{array} \right).
\end{equation}

From  the Proposition (\ref{properties of subspaces of tupe $U^{IJ}$, $U^{IK}$, $U^{JK}$}) point (\ref{1:1 correspondence between standard subspaces of $U^{IJ}(X_1)$ and associated subspaces of type $U^{IK}$} ) we know that, by changing leading vector  with  $ X \in L(X_5,X_6)$, one has  that $U^{IJ}(X)=U^{IJ}(X_5)$ and $U^{IK}(X)=U^{IK}(X_5)$. We then consider a new leading vector $ \bar X_5 \in L(X_5,X_6)$ and let $(\bar X_5,\bar X_6, \bar X_7, \bar X_8)$ and $(\bar X_5,\bar  X_6,\tilde{\bar{X_7}},\tilde{\bar{X_8}})$ be the new chains centered on $\bar X_5$ determined in order that the pair $(X_3, \tilde{\bar {X_7}})$ are related principal vectors of the pair $(L(X_3,X_4), L(\tilde X_7, \tilde X_8))$.
The  change of basis in $L(X_5,X_6)$  to diagonalize the minors $A_{12,34}$  and $A_{34,12}$ is given by the orthogonal transformation $(\bar X_5,\bar X_6)= C(X_5,X_6)$  represented by the matrix
\[
C:
\frac{1}{\sqrt{a^2+ b^2}}
\left(
\begin{array} {ll}
a & -b\\
b & a
\end{array}
\right)
\]
which induce  the same transformation onto $L(X_7,X_8)$ and $L(\tilde X_7,\tilde X_8)$.

The Gram matrix w.r.t. the bases $(X_3,X_4, \bar X_7, \bar X_8)$ and $(\tilde{X_3},\tilde{X_4},\bar{\tilde  {X_7}}, \bar{\tilde{X_8}})$ is the following
\[
\left(
\begin{array} {cc|cc}
 \Gamma & -\Delta &  +\sqrt{1- \Gamma^2 - \Delta^2} & 0   \\
 \Delta & \Gamma & 0 & -\sqrt{1- \Gamma^2 - \Delta^2}  \\
\hline
 -\sqrt{1- \Gamma^2 - \Delta^2} & 0 & \Gamma & -\Delta \\
  0 &  +\sqrt{1- \Gamma^2 - \Delta^2} & \Delta & \Gamma
\end{array} \right).
\]

 By construction the  vector $ \tilde{\bar{X_7}}$  is  the principal vector  associated to $X_3$ of the pair of isoclinic subspaces $(L(X_3,X_4),L(\tilde X_7,\tilde X_8))$ and in fact $\cos \widehat{X_3,\tilde{\bar{X_7}}} \geq 0$  whereas,  being  $\cos \widehat{X_4,\tilde{\bar{X_8}}} \leq 0$  the pair $(X_4,\tilde{\bar{X_8}})$  is not formed by  related principal vectors.

Let $(\bar X_5, \bar X_6,\bar X_7,\bar X_8)$  and $(\bar X_5, \bar Y_6,\bar X_7,\bar Y_8)$ the $\omega^I$ and $\omega^J$ chains of $U^{IJ}(\bar X_5)$ and $(\bar X_5, \bar X_6, \tilde{\bar {X_7}},\tilde{\bar {X_8}})$  and $(\bar X_5, \bar Z_6, \tilde{\bar {X_7}},\bar Z_8)$ the  $\omega^I$ and $\omega^K$ chains of $U^{IK}( \bar X_5)$.

W.r.t. the bases $(X_1,X_2,X_3, X_4,\bar X_5,\bar X_6,\bar X_7,\bar X_8)$ and $(X_1, Z_2,\tilde X_3,Z_4, \bar X_5,\bar Z_6,\tilde{\bar{X_7}},\bar Z_8)$, and denoting by $\Sigma=\sqrt{1- \Gamma^2 - \Delta^2}$  the matrix $C_{IK}$ is given by
\begin{equation} \label{$C_{IK}$ block of an 8 dimensional isoclinic addend}
C_{IK}=\left(
\begin{array} {cccc|cccc}
1 & 0 & 0 & 0 & 0 & 0 & 0 & 0 \\
0 & \chi & 0 & - \sqrt{1- \chi^2} & 0 & 0 & 0 & 0 \\
0 & -\Delta\sqrt{1- \chi^2} & \Gamma & -\Delta  \chi & 0 & 0 &  \Sigma  &  0  \\
0 & \Gamma \sqrt{1- \chi^2} & \Delta  & \Gamma \chi & 0 &  \Sigma  \sqrt{1- \chi^2} & 0 & - \Sigma \chi\\
\hline
0 & 0 & 0 & 0 & 1 & 0 & 0 & 0 \\
0 & 0 &  0 & 0 & 0 & \chi & 0 & - \sqrt{1- \chi^2} \\
0 & 0 &  -\Sigma  &  0 &  0 & -\Delta \sqrt{1- \chi^2} & \Gamma & -\Delta  \chi \\
0 &  \Sigma \sqrt{1- \chi^2}  & 0 &  \Sigma  \chi & 0 & \Gamma \sqrt{1- \chi^2} & \Delta  & \Gamma \chi
\end{array} \right).
\end{equation}

W.r.t. the $\omega^I$-standard basis  $(X_1,X_2,X_3, X_4,\bar X_5,\bar X_6,\bar X_7,\bar X_8)$ and  the  $\omega^J$-standard basis  $(\bar X_1,Y_2,X_3=Y_3, Y_4,X_5,Y_6, X_7,Y_8)$ the matrix $C_{IJ}$ is given by
 \begin{equation}   \label{$C_{IJ}$ block of an 8 dimensional isoclinic addend}
C_{IJ}=\left(
\begin{array} {cccc|cccc}
1 & 0 & 0 & 0 & 0 & 0 & 0 & 0\\
0 & \xi & 0 & -\sqrt{1-\xi^2} & 0 & 0 & 0 & 0\\
0 & 0 & 1 & 0 & 0 & 0 & 0 & 0\\
0 & \sqrt{1-\xi^2} & 0 & \xi & 0 & 0 & 0 & 0\\
\hline
0 & 0 & 0 & 0 & 1 & 0 & 0 & 0\\
0 & 0 & 0 & 0 & 0 & \xi & 0 & -\sqrt{1-\xi^2}\\
0 & 0 & 0 & 0 & 0 & 0 & 1 & 0 \\
0 & 0 & 0 & 0 & 0 & \sqrt{1-\xi^2} & 0 & \xi
\end{array}
\right).
\end{equation}

The aforementioned $\omega^I$ (resp. $\omega^K$)-standard bases are union of a pair of $\omega^I$ (resp. $\omega^K$)-chains of $U_1^8$  generating  a pair of orthogonal subspaces of type $U^{IJ}$  (resp. $U^{IK}$).
 The independence of $(\xi,\chi,\eta,\Gamma,\Delta)$ from the leading vector,   accounts for the independence of both matrices from the  leading vector $X_1 \in U_1^8$. We call  \textit{canonical bases} of $U_1^8$ the set of such bases determined by the leading vector $X_1 \in U_1^8$ and \textit{canonical matrices $C_{IJ}$ and $C_{IK}$} the  matrices given in (\ref{$C_{IJ}$ block of an 8 dimensional isoclinic addend}) and  (\ref{$C_{IK}$ block of an 8 dimensional isoclinic addend}) respectively..


We can continue this construction considering an unitary vector  $\tilde X_1 \in (U_1^8)^\perp$  which will determine a second 8-dimensional subspace $U_2^8$ orthogonal to the first and so on.
Then we can state the following
\begin{prop}
 Any $U \in \mathcal{IC}^{8k}$ admits an  orthogonal decomposition
    \[ U= U_1^8  \stackrel{\perp}{\oplus} U_2^8  \stackrel{\perp}{\oplus} U_3^8 \stackrel{\perp}{\oplus}  \ldots  \stackrel{\perp}{\oplus} U_k^8.\]
    where  all 8-dimensional addends are isoclinic i.e. $U_i \in \mathcal{IC}^{8}, \, i= 1, \ldots,k$ with same angles for all pairs $(U_i,AU_i)$ as $(U,AU)$.
   \end{prop}

\begin{prop}
To any $U \in \mathcal{IC}^{8k}$  we can associate the values $\Gamma$ and $\Delta$ of all 4-dimensional subspaces of type $U^{IJ}$ belonging to  $U$.
\end{prop}

We can then state the
\begin{coro} \label{matrices $C_{IJ}$ and $C_{IK}$ of an isoclinic subspace}
  Let $U \in \mathcal{IC}^{8k}$.
For the  canonical matrices one has $C_{IJ}=\bigoplus_{i=1}^k (C_{IJ})_i$ and $C_{IK}=\bigoplus_{i=1}^k (C_{IK})_i$  where the blocks are given in  (\ref{$C_{IJ}$ block of an 8 dimensional isoclinic addend})  and in  (\ref{$C_{IK}$ block of an 8 dimensional isoclinic addend}) respectively. Such matrices  depend neither on the particular decomposition of $U$ into isoclinic 8-dimensional subspaces nor  on the canonical bases chosen for each addend.
\end{coro}

So we proved that if $\dim U= 8k$, $U$ admits an orthogonal  decomposition into 8-dimensional isoclinic subspaces $U_i$ whose angles of isoclinicity are   $\theta^I,\theta^J,\theta^K$.
According to the Theorem (\ref{main_theorem 1_ existance of canonical bases with same mutual position})  and Corollary (\ref{matrices $C_{IJ}$ and $C_{IK}$ of an isoclinic subspace}) the orbit is determined by the triple $(\theta^I,\theta^J,\theta^K)$, by the triple $(\xi,\chi,\eta)$ and by $\Delta$ (recall that $\Gamma$ is a function of $(\xi,\chi,\eta)$).


If $\dim U= 8k+2$ or  $\dim U= 8k+6$ necessarily $\xi= \pm 1, \chi= \pm 1, \eta= \xi  \cdot \chi$ and $U$ decomposes into orthogonal sum of 2 dimensional isoclinic addends $U_i$.   For any $A \in S(\mathcal{Q})$ the angle of isoclinicity  of the pair of standard 2-planes  $(U_i,AU_i)$ is the same as the one of the pair $(U,AU)$. From the Proposition (\ref{Gamma is an invariant of an isoclinic 4 dimensional subspace 1}) one has  $\Gamma=1$. The orbit in this case is determined by the triple $(\theta^I,\theta^J,\theta^K)$ and by the pair $(\xi,\chi)$.

 Finally if $\dim U= 8k+4$, then $U$ is direct orthogonal sum of isoclinic 4-dimensional subspaces $U_i$ in which case $\Gamma^2 + \Delta^2=1$. For any $A \in S(\mathcal{Q})$ the angle of isoclinicity  of the pair $(U_i,AU_i)$ is the same as the one of the pair $(U,AU)$.
 The orbit is determined by the triple $(\theta^I,\theta^J,\theta^K)$, by the triple $(\xi,\chi,\eta)$ and by the sign of $\Delta=\pm \sqrt{1-\Gamma^2}$ with $\Gamma=\Gamma(\xi,\chi,\eta)$.
Recalling that the canonical matrices are given in (\ref{$C_{IJ}$ block of an 8 dimensional isoclinic addend}) and (\ref{$C_{IK}$ block of an 8 dimensional isoclinic addend}), we can now state the final result.

\begin{teor}\label{Orbit of an isoclinic subspace}
Let  $U \in \mathcal{IC}^{2m}$. Let fix an admissible basis $(I,J,K)$  and denote by  $(\theta^I,\theta^J,\theta^K)$ the  angles of isoclinicity of the pairs $(U,IU),(U,JU),(U,KU)$ respectively. For $k \geq 0$:
\begin{itemize}
\item If $2m= 8k+2$ or  $2m= 8k+6$, $U$ is 2-planes decomposable i.e. is orthogonal sum of $U_i \in  \mathcal{IC}^2$ with same angle of isoclinicity of $U$.
In this case $(\Gamma,\Delta)=(1,0)$ and  the pair  $(\xi,\chi)=(\pm 1, \pm 1)$   determine the matrices $C_{IJ},C_{IK}$.
The $Sp(n)$-orbit is then determined by the angles $(\theta^I,\theta^J,\theta^K)$ and by the pair $(\xi,\chi)$.
\item If $2m= 8k+4$, then $U$  is orthogonal sum of $U_i \in  \mathcal{IC}^4$ with same angle of isoclinicity of $U$ and characterized by the same pair $(\Gamma,\Delta)$. In this case $\Gamma^2+\Delta^2=1$ and the canonical matrices are determined by $(\xi,\chi,\Gamma,\Delta)$.  In particular this case always occurs if $U$ is orthogonal in which case $(\Gamma,\Delta)=(1,0)$.
    The $Sp(n)$-orbit is then characterized by   $(\theta^I,\theta^J,\theta^K)$ and $(\xi,\chi,\eta,\Delta)$.
    In particular, if $\xi=\pm 1$ and $\chi=\pm 1$ we are  in the first case.
\item If $2m= 8k$ then   $U$  is orthogonal sum of $U_i \in  \mathcal{IC}^8$ with same angle of isoclinicity of $U$. The canonical matrices are determined by $(\xi,\chi,\Gamma,\Delta)$  where $\Gamma^2 + \Delta^2 \leq 1$ and the $Sp(n)$-orbit by $(\theta^I,\theta^J,\theta^K)$ and $(\xi,\chi,\eta,\Delta)$. If in particular $\Gamma^2 + \Delta^2=1$ we are in the previous case and if
     furthermore  $\xi=\pm 1$ and $\chi=\pm 1$ we are  in the first case.
\end{itemize}
\end{teor}

\end{document}